\documentclass[10pt]{amsart}
\usepackage{amsmath,amssymb,latexsym,cancel,rotating}
\usepackage{caption}
\usepackage{graphicx,amssymb,mathrsfs,amsmath,color,fancyhdr,amsthm}
\usepackage[all]{xy}
\usepackage{setspace}

\newtheorem{theorem}{Theorem}[section]

\newtheorem{corollary}[theorem]{Corollary}

\newtheorem{proposition}[theorem]{Proposition}

\theoremstyle{definition}
\newtheorem{remark}[theorem]{Remark}

\begin{document}
\title[MHM on quivers and preprojective KHA]
{Associated graded modules of mixed Hodge modules on quivers and preprojective K-theoretic Hall algebra}
\author[Lan]{Yixin Lan}

\address{Max Planck Institute for mathematics,Vivatsgasse,Bonn, 53111,North Rhine-Westphalia,German}
\email{lanyx1997@mpim-bonn.mpg.de (Y.Lan)}

\bibliographystyle{abbrv}

\begin{abstract}
	We study the connection between mixed Hodge modules on quiver moduli spaces and the K-theoretic Hall algebra of preprojective algebras by taking associated graded modules. Furthermore, we show that the functor taking associated graded modules commutes with reflection functors.
\end{abstract}
\maketitle

\begin{spacing}{1}
\section{Introduction}
\subsection{Quivers and quantum groups}
Given a quiver $Q=(Q_{0},Q_{1},s,t)$ which allows loops and multiple arrows, where $Q_{0}$ is the set of vertices, $Q_{1}$ is the set of arrows and $s,t:Q_{1} \rightarrow Q_{0}$ are maps taking the source and target respectively, the Cartan matrix  $A=(a_{i,j})_{i,j \in Q_{0}}$ of $Q$ is given by $a_{i,j}=2\delta_{i,j}- |\{e \in Q_{1}| \{s(e),t(e)\}= \{ i,j \}| $.  The Cartan matrix defines a bilinear form of $\mathbb{Z}Q_{0}$, which is called the symmetric Euler form, and we denote it by $(-,-)$.  If there is no loop at $i \in Q_{0}$, we call $i$ a real vertex, otherwise we call $i$ an imaginary vertex. Denote the set of real and imaginary vertices by $Q^{re}_{0}$ and $Q^{im}_{0}$ respectively, then there is a generalized Kac-Moody algebra $\mathbf{U}_{v}(\mathfrak{n}^{+}_{Q})$ (which is also called Borcherds-Bozec algebra) considered in \cite{Bo15}, whose positive simple roots are parameterized by $I_{\infty}=(Q^{re}_{0}\times \{1\} ) \coprod (Q^{im}_{0}\times \mathbb{Z}_{\geqslant 1}). $ More precisely, $\mathbf{U}_{v}(\mathfrak{n}^{+}_{Q})$ is the $\mathbb{Q}(v)$-algebra generated by $E_{i'}, i' \in I_{\infty}$ with the following relations
\begin{enumerate}
	\item  For $i',j' \in I_{\infty}$, $[E_{i'},E_{j'}]=0$ if $(i',j')=0$.  Here $(-,-)$ is defined by $((i,n),(j,m) )=mn (i,j),$ where the second $(-,-)$ is the symmetric Euler form.
	\item  For $j' \in Q^{re}_{0} \times \{1\}$ and $ i'\neq j' \in I_{\infty}$,
	$$ \sum\limits_{k+l= 1-(j',i')} (-1)^{k} \begin{pmatrix}
		1- (j',i')\\ k
	\end{pmatrix}_{v}   E^{k}_{j'}E_{i'}E^{l}_{j'}=0 ,$$
	 where the quantum binomial coefficients are defined by
	$$[n]_{v}=\frac{v^{n}-v^{-n}}{v-v^{-1}}, \quad [n]_{v}!=\prod\limits_{m=1}^{n} [m]_{v}, \quad \begin{pmatrix}
		m\\ n
	\end{pmatrix}_{v} = \frac{[m]_{v}!}{[n]_{v}! [m-n]_{v}!}.  $$
\end{enumerate}  
The Lusztig integral form $\mathbf{U}^{\mathbb{Z}}_{v}(\mathfrak{n}^{+}_{Q})$ is the $\mathbb{Z}[v^{\pm}]$ subalgebra generated by $E_{i'},i' \in I_{\infty}$ and $E^{(a)}_{i}= \frac{E^{a}_{i}}{[a]_{v}!},  a \geqslant 1, i \in Q^{re}_{0}$. Here we denote $E_{i,1}$ by $E_{i}$ if $i$ is a real vertex.

This paper focuses on the geometric realizations of the inetgral form $\mathbf{U}^{\mathbb{Z}}_{v}(\mathfrak{n}^{+}_{Q})$ and the comparison between them, and similar problem for enveloping algebras has been studied in \cite{HLCC}.  The geometric realizations of $\mathbf{U}^{\mathbb{Z}}_{v}(\mathfrak{n}^{+}_{Q})$ are provided by
\begin{enumerate}
	\item Certain mixed Hodge modules on the stack of representations of $Q$. (See \cite{achar2021perverse} for mixed Hodge modules of quivers of finite type, \cite{L90} \cite{Lu91} and \cite{Bo15} for underlying constructible sheaves of general quivers.)
	\item Certain subalgebra of the K-theoretic Hall algebra of the strongly seminilpotent stack of the preprojective algebra $\Pi_{Q}$ of $Q$. (See \cite{Tu23} for general quivers with potentials and \cite{KHAVV} for preprojective algebras.)
\end{enumerate}

\subsection{Reflection functors and Lusztig symmetries}
For a real vertex $i \in Q_{0}$, Lusztig symmetries $T'_{i,-1},T''_{i,1}$ induce linear isomorphisms between certain quotient $\mathbb{Z}[v^{\pm}]$-modules $$^{(i)}\mathbf{U}^{\mathbb{Z}}_{v}(\mathfrak{n}^{+}_{Q} )= \mathbf{U}^{\mathbb{Z}}_{v}(\mathfrak{n}^{+}_{Q} )/\sum\limits_{a\geqslant 1}E^{(a)}_{i}\mathbf{U}^{\mathbb{Z}}_{v}(\mathfrak{n}^{+}_{Q} ), $$
$$ \mathbf{U}^{\mathbb{Z},(i)}_{v}(\mathfrak{n}^{+}_{Q} )= \mathbf{U}^{\mathbb{Z}}_{v}(\mathfrak{n}^{+}_{Q} )/\sum\limits_{a\geqslant 1}\mathbf{U}^{\mathbb{Z}}_{v}(\mathfrak{n}^{+}_{Q} )E^{(a)}_{i}.$$
Taking $j' \in I_{\infty}$ such that $j' \neq (i,1)$ and $m \in \mathbb{Z}$, the formulas of Lusztig symmetries are given by 
$$ T''_{i,1} (f (i,j',m) ) =f'(i,j',-m-(i,j') ),$$ 
$$T'_{i,-1}(f'(i,j',m))=  f(i,j',-m-(i,j') ),$$
where  $$f(i,j',m)= \sum\limits_{k+l= m} (-1)^{k} v^{-k(1-m-(i,j'))} E^{(k)}_{i}E_{j'} E^{(l)}_{i}, $$
$$f'(i,j',m)= \sum\limits_{k+l= m} (-1)^{k} v^{-k(1-m-(i,j'))} E^{(l)}_{i}E_{j'} E^{(k)}_{i}. $$
We remark that Lusztig defines $T'_{i,-1},T''_{i,1}$ as isomorphisms between submodules in \cite{Lu97} and \cite[38.1]{lusztig2010introduction}, but his definition also works for quotients.

 Another purpose of this paper is to study the realization of Lusztig symmetries in terms of geometric models and the comparison between them. The Lusztig symmetries are realized by 
\begin{enumerate}
	\item The isomorphism between stacks of representations of $Q$ induced from   BGP reflection functors. (See \cite{BGP} for underlying constructible sheaves.)
	\item The isomorphism between $K$-theory of  the strongly seminilpotent stack of the preprojective algebra $\Pi_{Q}$ arising from the tilting theory of  $\Pi_{Q}$ modules. (See \cite{SY13} and \cite{BKT14}). 
\end{enumerate}

\subsection{Main results}
\subsubsection{Geometric realization of the integral form}
For a dimension vector $\alpha \in \mathbb{N}Q_{0}$ of a fixed quiver $Q$, let $\mathbf{E}_{\alpha,Q}$ be the moduli space of representations of $Q$ with dimension vector $\alpha$, and the algebraic group $\mathbf{G}_{\alpha}$ acts on it. Let $\mathcal{Q}_{\alpha}$ be the category of (complexes of) $\mathbf{G}_{\alpha}$-equivariant mixed hodge modules whose underlying complexes of constructible sheaves are those semisimple complexes considered in \cite{Lu91} for loop free case and \cite{Bo15} for quivers with loops. (We remark that the objects considered in this paper differs from those in \cite{Bo15} by Fourier transforms, see Section \ref{BB}.)
\begin{theorem}\label{main1}
	For a given quiver $Q$, let $\mathfrak{N}_{\alpha}=\mathfrak{N}_{\alpha,Q}$ be the strictly seminilpotent stack of preprojective algebra $\Pi_{Q}$, then taking associated graded modules induces an injective morphism of $\mathbb{Z}[v^{\pm}]$-algebras 
	$$ [gr]: \mathcal{K}'_{\oplus} (\mathcal{Q})=\bigoplus\limits_{\alpha \in \mathbb{N}Q_{0}}\mathcal{K}'_{\oplus} (\mathcal{Q}_{\alpha}) \rightarrow  \mathcal{A}^{nil}= \bigoplus\limits_{\alpha \in \mathbb{N}Q_{0}}K_{\mathbb{C}^{\times}} ( \mathfrak{N}_{\alpha}), $$
	where $ \mathcal{K}'_{\oplus} (\mathcal{Q})$ has the (twisted) Lusztig's induction product and $\mathcal{A}^{nil}$ has the (restriction of) twisted K-theoretic Hall product.  The image of $[gr]$ is the subalgebra $\mathcal{A}^{nil}_{zs}$ generated by $[\mathcal{O}_{T^{\ast}_{0} \mathbf{E}_{ai} } ] ,i \in Q_{0},a\geqslant 1$, called the zero spherical subalgebra, which is isomorphic to $\mathbf{U}^{\mathbb{Z}}_{v}(\mathfrak{n}^{+}_{Q})$ via a canonical isomorphism.
\end{theorem}

The theorem above can be viewed as a quantization of \cite[Theorem 1.1 and 1.2]{HLCC}. More precisely, the proof of Proposition \ref{inj} shows the following results.
\begin{theorem}\label{main2}
	With the notations above, let  $\mathcal{K}'_{\oplus} (rat(\mathcal{Q}))$ be the (split) Grothendieck group of underlying semisimple complexes of objects in $\mathcal{Q}$ with the algebra structure given by (twisted) Lusztig's induction, and let $\mathcal{H}^{0}=\bigoplus\limits_{\alpha \in \mathbb{N}Q_{0}} {\rm{H}}^{BM}_{top}(\mathfrak{N}_{\alpha},\mathbb{Z})$ be the (twisted) zero degree cohomological Hall algebra of strongly seminilpotent stacks. After taking classical limit $v \rightarrow 1$, there is a commutative diagram of $\mathbb{Z}$-algebras
	\[
	\xymatrix{   \mathcal{K}'_{\oplus} (\mathcal{Q})|_{v=1} \ar[r]^{[gr]}  \ar[d]_{[rat]} &  \mathcal{A}^{nil}_{zs}|_{v=1} \ar[d]^{Ch_{0}} \\
	\mathcal{K}'_{\oplus} (rat(\mathcal{Q}))|_{v=1} \ar[r]^-{CC}	&     \mathcal{H}^{0},
	}
	\]
	where $[rat]$ takes underlying constructible sheaves of mixed Hodge modules,
	$CC$ is the characteristic cycle map considered in \cite{HLCC}, and $Ch_{0}$ is the map taking zero degree Chern characters. All algebras above are canonically  isomorphic to the enveloping algebra  $\mathbf{U}^{\mathbb{Z}}(\mathfrak{n}^{+}_{Q})$.
\end{theorem}

\subsubsection{Geometric realization of Lusztig symmetries}
Fix $i \in Q^{re}_{0}$ and assume that $i$ is a source in $Q$, let $Q'$ be the quiver obtained by reversing all arrows of $Q$ at $i$. Take dimension vectors $\alpha,\beta$ such that $s_{i}(\alpha)=\beta$, where $s_{i}$ is the simple reflection. Let $\mathbf{E}^{(i)}_{\alpha,Q,0}/\mathbf{G}_{\alpha}, {^{(i)}  \mathbf{E}_{\beta,Q',0}}/\mathbf{G}_{\beta},\mathfrak{N}^{(i)}_{\alpha,Q} $ and $^{(i)}\mathfrak{N}_{\beta,Q'}$ be the open substacks of $\mathbf{E}_{\alpha,Q}/\mathbf{G}_{\alpha},\mathbf{E}_{\beta,Q'}/\mathbf{G}_{\beta}, \mathfrak{N}_{\alpha,Q}$ and $\mathfrak{N}_{\beta,Q'}$ respectively, which parameterize the representations (of quivers or their preprojective algebras) having no direct summand isomorphic to the simple module of dimension vector $i$. 

The BGP reflection functor of quivers induces an isomorphism between the open substacks $\mathbf{E}^{(i)}_{\alpha,Q,0}/\mathbf{G}_{\alpha}$ and $ {^{(i)}  \mathbf{E}_{\beta,Q',0}}/\mathbf{G}_{\beta}$, and hence defines an isomorphism $[\mathbf{S}_{i}]$ between the (split) Grothendieck groups of mixed Hodge modules on these substacks. Similarly, the tilting theory of preprojective algebra provides an isomorphism between $\mathfrak{N}^{(i)}_{\alpha,Q} $ and $^{(i)}\mathfrak{N}_{\beta,Q'}$, and induces an isomorphism $[\mathbf{R}_{i}]$ between their $K$ groups.

\begin{theorem}\label{main3}
	Denote the restriction of $ K'_{\oplus} (\mathcal{Q}_{\alpha,Q}), K'_{\oplus} (\mathcal{Q}_{\beta,Q'}), \mathcal{A}^{nil}_{\alpha} $ and $\mathcal{A}^{nil}_{\beta}$ on the associated open substacks by $K'_{\oplus} (\mathcal{Q}^{(i)}_{\alpha,Q}),K'_{\oplus} ({^{(i)}\mathcal{Q}}_{\beta,Q'}),\mathcal{A}^{nil,(i)}_{\alpha}   $ and $^{(i)}\mathcal{A}^{nil}_{\beta}$ respectively, then there is a commutative diagram of $\mathbb{Z}[v^{\pm}]$-modules
	\[
	\xymatrix{
		K'_{\oplus} (\mathcal{Q}^{(i)}_{\alpha,Q})  \ar[r]^{[\mathbf{S}_{i}]} \ar[d]_{[gr]}	& K'_{\oplus} ({^{(i)}\mathcal{Q}}_{\beta,Q'}) \ar[d]^{[gr]} \\
		\mathcal{A}^{nil,(i)}_{\alpha} \ar[r]^{[\mathbf{R}_{i}]}	&  {^{(i)}\mathcal{A}}^{nil}_{\beta}.
  	}
	\]
	Moreover, after restricting to zero spherical subalgebras, we obtain the following commutative diagram of isomorphisms
	\[
	\xymatrix{
		& \mathcal{K}'_{\oplus}( \mathcal{Q}^{(i)}_{\alpha,Q} ) \ar[r]^{[\mathbf{S}_{i}]} \ar[dd] &  \mathcal{K}'_{\oplus}({^{(i)}\mathcal{Q}}_{\beta,Q'}) \ar[dd] & \\
		\mathbf{U}^{\mathbb{Z},(i)}_{v} (\mathfrak{n}^{+}_{Q}) \ar[rrr]^{T''_{i,1}} \ar[ru]^{\phi_{Q}} \ar[rd]_{\psi_{Q}}	& & & ^{(i)}\mathbf{U}^{\mathbb{Z}}_{v} (\mathfrak{n}^{+}_{Q'}) \ar[lu]_{\phi_{Q'}} \ar[ld]^{\psi_{Q'}}  \\
		&  \mathcal{A}^{nil,(i)}_{zs,\alpha} \ar[r]^{\mathbf{R}_{i}} &  ^{(i)}\mathcal{A}^{nil}_{zs,\beta} &.
	}
	\]
 \end{theorem}

\subsection{Structure of the paper}
In section 2, we define the functor taking associated graded modules of equivariant derived categories of mixed Hodge modules and study its functorial properties. In section 3, we compare the geometric realizations of Lusztig integral form via the functor $gr$. In section 4, we study the connection between the geometric BGP reflection functors and the derived reflection functors for K theoretic Hall algebra of preprojective algebras.

\subsection*{Convention}
Given a smooth complex variety $X$, we denote its dimension by $d_{X}$, the tangent sheaf by $\Theta_{X}$, the cotangent sheaf by $\Omega_{X}$, and $\Omega_{X}^{d_{X}}$ by $\omega_{X}$. Since we usually work on derived categories, we denote by $f^{\ast},f_{\ast},f^{!},f_{!},\otimes$ their derived functors for simplicity. We use right $\mathcal{D}$-modules in the theory of mixed Hodge modules, following  \cite{Schnell2014AnOO} and the proofs in \cite{saito1990mixed}. For example, the pure mixed Hodge module $\underline{\mathbb{Q}}^{H}_{X}$ of weight $0$ for a smooth complex variety $X$ is given by $\underline{\mathbb{Q}}^{H}_{X}=(\omega_{X},F_{\bullet} \omega_{X},\underline{\mathbb{Q}}_{X}[{\rm{dim}}X])$, where $\omega_{X}$ is regarded as a right $\mathcal{D}_{X}$-module and $gr^{F}_{p}\omega_{X}=0$ unless $p=- {\rm{dim}}X $.
\section{Mixed Hodge modules and associated graded modules}
\subsection{Equivariant mixed Hodge modules}
Let $X$ be a smooth complex variety and $\mathbf{k}$ be a subfield of $\mathbb{R}$, we  consider a quadruple $(\mathcal{M},F,\mathcal{K},W)$, where
\begin{enumerate}
    \item $\mathcal{M}$ is a regular holonomic right
     $\mathcal{D}_{X}$-module;
    \item $F=F_{\bullet}\mathcal{M}$ is a good filtration on $\mathcal{M}$, which is called the Hodge filtration;
    \item $\mathcal{K}$ is a perverse sheaf in ${\rm{Perv}}(X,\mathbf{k})$ equipped with an isomorphism $\mathbb{C}\otimes_{\mathbf{k}} \mathcal{K} \cong \rm{dR}(\mathcal{M})$, where $\rm{dR}$ is the de Rham functor in Riemann-Hilbert correspondence;
    \item $W=W_{\bullet}\mathcal{K}$  is an increasing filtration on $\mathcal{K}$, which is called the weight filtration.
\end{enumerate}
Quadruples $(\mathcal{M},F,\mathcal{K},W)$ satisfying certain conditions are called  mixed Hodge modules on $X$, and they form a finite length Abelian category ${\rm{MHM}}(X,\mathbf{k})$. One can see details in \cite{Saito1988MHM}, \cite{saito1990mixed} and \cite{Schnell2014AnOO}.

There is an exact faithful functor 
$ {\rm{rat}}:{\rm{MHM}}(X,\mathbf{k}) \longrightarrow  {\rm{Perv}}(X,\mathbf{k}) $
defined by $(\mathcal{M},F,\mathcal{K},W) \mapsto \mathcal{K}$, and this functor also induces a functor of derived categories $${\rm{rat}}:\mathcal{D}^{b}({\rm{MHM}}(X,\mathbf{k}) )  \longrightarrow \mathcal{D}^{b}( {\rm{Perv}}(X,\mathbf{k}) ) \cong \mathcal{D}^{b}_{c}(X,\mathbf{k}).$$

Mixed Hodge modules admit a six functor formalism. Concretely, for a morphism $f:X \rightarrow Y$ of complex varieties, there are functors
$$ f_{\ast},f_{!}:   \mathcal{D}^{b}({\rm{MHM}}(X,\mathbf{k}) )   \longrightarrow \mathcal{D}^{b}({\rm{MHM}}(Y,\mathbf{k}) ), $$
$$ f^{\ast},f^{!}:   \mathcal{D}^{b}({\rm{MHM}}(Y,\mathbf{k}) )   \longrightarrow \mathcal{D}^{b}({\rm{MHM}}(X,\mathbf{k}) ), $$
$$ \otimes: \mathcal{D}^{b}({\rm{MHM}}(X,\mathbf{k}) ) \times \mathcal{D}^{b}({\rm{MHM}}(X,\mathbf{k}) ) \longrightarrow \mathcal{D}^{b}({\rm{MHM}}(X,\mathbf{k}) ),   $$
$$ R\mathcal{H}om : \mathcal{D}^{b}({\rm{MHM}}(X,\mathbf{k}) )^{op} \times \mathcal{D}^{b}({\rm{MHM}}(X,\mathbf{k}) ) \longrightarrow \mathcal{D}^{b}({\rm{MHM}}(X,\mathbf{k}) )$$
that commute with $rat$. These functors satisfy natural adjoint property and base change property.

\subsection{Tate twist and its square root}
When $X=pt$, the Tate module $\mathbf{k}(1)=(\mathbb{C},F,2\pi \sqrt{-1} \mathbf{k})\in {\rm{MHM}}(pt,\mathbf{k})$ determines a pure Hodge module of weight $-2$. For general $X$, tensoring with the Tate module defines a functor $(1): \mathcal{D}^{b}({\rm{MHM}}(X,\mathbf{k}) )  \rightarrow  \mathcal{D}^{b}({\rm{MHM}}(X,\mathbf{k}) )$, which is called the Tate twist.

In order to construct self-dual IC complex, we need to add the square root of the Tate twist, following \cite{beilinson1996koszul}. Let $\mathcal{D}^{b}({\rm{MHM}}(X,\mathbf{k}) )'$ be the category $\mathcal{D}^{b}({\rm{MHM}}(X,\mathbf{k}) ) \bigoplus \mathcal{D}^{b}({\rm{MHM}}(X,\mathbf{k}) )$ and let $(\frac{1}{2})$ be the functor sending $\mathcal{F}=(\mathcal{F}_{1},\mathcal{F}_{2})$ to $(\mathcal{F}_{2}(1),\mathcal{F}_{1})$. Regard $\mathcal{D}^{b}({\rm{MHM}}(X,\mathbf{k}) )$ as the full subcategory of $\mathcal{D}^{b}({\rm{MHM}}(X,\mathbf{k}) )'$ via the diagonal map, then it's easy to check that $(\frac{1}{2})^{2}=(1)$ on $\mathcal{D}^{b}({\rm{MHM}}(X,\mathbf{k}) )$. The six functors are also naturally defined on $\mathcal{D}^{b}({\rm{MHM}}(X,\mathbf{k}) )'$. 

The Grothendieck group $\mathcal{K}_{0}  ( \mathcal{D}^{b}({\rm{MHM}}(X,\mathbf{k}) )') $ of $\mathcal{D}^{b}({\rm{MHM}}(X,\mathbf{k}) )'$ is the $\mathbb{Z}[v^{\pm}]$-module spanned by  $\{[\mathcal{M}],\mathcal{M}$ is an object of $\mathcal{D}^{b}({\rm{MHM}}(X,\mathbf{k}) )'\}$, and subject to the relation   
$$[\mathcal{M}]+[\mathcal{L}]=[\mathcal{N}],  $$ 
if $\mathcal{M} \rightarrow \mathcal{N} \rightarrow \mathcal{L} \xrightarrow{+1}$ is a distinguished triangle in $\mathcal{D}^{b}({\rm{MHM}}(X,\mathbf{k}) )'$, and
$$[\mathcal{M}(\frac{1}{2})]=v[\mathcal{M}].$$

\subsection{Associated graded module of mixed Hodge module}\label{grdef}
\subsubsection{Definition of $gr$} \label{grt}
Given a mixed Hodge module on $X$, one can take its associated graded module to get a $\mathbb{C}^{\times}$-equivariant quasi-coherent sheaf on the cotangent bundle $T^{\ast}X$. Here $\mathbb{C}^{\times}$ acts on $T^{\ast}X$ by scaling the fibers with weight $-1$.

Recall that $\mathcal{D}_{X}$ is equipped with a natural filtration given by the order of differential operator, and its associated graded algebra $gr \mathcal{D}_{X}$ is isomorphic to $\pi_{\ast} \mathcal{O}_{T^{\ast}X}$, where $\pi: T^{\ast}X \rightarrow X$ is the projection.  For a $\mathcal{D}_{X}$-module $\mathcal{M}$ with compatible filtration $F=F_{\bullet}\mathcal{M}$, the associated graded module $gr(\mathcal{M})$ becomes a $\pi_{\ast} \mathcal{O}_{T^{\ast}X}$-module.  Since the localization gives an equivalence between $\pi_{\ast} \mathcal{O}_{T^{\ast}X}$-modules and  $ \mathcal{O}_{T^{\ast}X}$-modules, one can regard $gr(\mathcal{M})$ as a $\mathbb{Z}$-graded coherent sheaf on $T^{\ast}X$ and obtain a functor 
$$ gr:\mathcal{D}^{b}({\rm{MHM}}(X,\mathbf{k}) ) \longrightarrow \mathcal{D}_{gr}^{b}(\mathcal{O}_{T^{\ast}X}-mod),$$
where we identify left $\mathcal{O}_{T^{\ast}X} $ or $\pi_{\ast} \mathcal{O}_{T^{\ast}X} $ modules with right $\mathcal{O}_{T^{\ast}X}$ or $\pi_{\ast} \mathcal{O}_{T^{\ast}X}$ modules respectively. This is  the same as a functor 
$$ gr:\mathcal{D}^{b}({\rm{MHM}}(X,\mathbf{k}) ) \longrightarrow \mathcal{D}^{b}({\rm{Coh}}^{\mathbb{C}^{\times}}(T^{\ast} X) ) \subseteq \mathcal{D}^{b}({\rm{Qcoh}}^{\mathbb{C}^{\times}}(T^{\ast} X) ).$$

After extending by $gr(\mathcal{M} (\frac{1}{2}))=gr(\mathcal{M})\langle 1\rangle$, where the $\langle 1\rangle$ means a square root of grading shift, we obtain a functor  
$$ gr:\mathcal{D}^{b}({\rm{MHM}}(X,\mathbf{k}) )' \longrightarrow \mathcal{D}^{b}({\rm{Coh}}^{\mathbb{C}^{\times}}(T^{\ast} X) )' \subseteq \mathcal{D}^{b}({\rm{Qcoh}}^{\mathbb{C}^{\times}}(T^{\ast} X) )',$$
where $\mathcal{D}^{b}({\rm{Coh}}^{\mathbb{C}^{\times}}(T^{\ast} X) )$ and $\mathcal{D}^{b}({\rm{Qcoh}}^{\mathbb{C}^{\times}}(T^{\ast} X) )$ are the bounded derived category of $\mathbb{C}^{\times}$-equivariant coherent and quasi-coherent sheaves on $T^{\ast}X$ respectively, $\mathcal{D}^{b}({\rm{Coh}}^{\mathbb{C}^{\times}}(T^{\ast} X) )'$ and $\mathcal{D}^{b}({\rm{Qcoh}}^{\mathbb{C}^{\times}}(T^{\ast} X) )'$ means that we have added the square root $\langle 1 \rangle$.

The  equivariant $K$ group  $K_{\mathbb{C}^{\times}} (T^\ast X)$ of $T^{\ast} X$ has a natural $\mathbb{Z}[v^{\pm}]$-module such that $\langle 1 \rangle$ acts by a square root $v=q^{\frac{1}{2}}$ of $q$. By definition of $gr$, the functor $gr$ induces a morphism of $\mathbb{Z}[v^{\pm}]$-modules
$$ [gr]: \mathcal{K}_{0}( \mathcal{D}^{b}({\rm{MHM}}(X,\mathbf{k}) )')  \longrightarrow K_{\mathbb{C}^{\times}} (T^\ast X).  $$

\subsubsection{Functorial property of $gr$}
Consider a smooth morphism $f:X \rightarrow Y$ of smooth complex varieties of relative dimension $d$, there is a correspondence of cotangent bundles

\[
\xymatrix{
	T^{\ast}X \ar[d]_{\pi_{X}}
	& X \times_{Y} T^{\ast}Y \ar[l]_{df^{t}} \ar[r]^{p_{1}} \ar[dl]_{p_{1}} \ar[d]^{p_{2}}
	& X \ar[d]^{f} \\	
	X
	& T^{\ast}Y \ar[r]^{\pi_{Y}}
	& Y.
}
\]
\begin{proposition} \label{pb}
    With the notations above, for any object $\mathcal{M}$ in  $\mathcal{D}^{b}({\rm{MHM}}(Y,\mathbf{k}) )$ or $\mathcal{D}^{b}({\rm{MHM}}(Y,\mathbf{k}) )'$, one has the following isomorphism $$gr(f^{\ast}\mathcal{M} ) \cong df^{{t}}_{\ast} p^{\ast}_{2}(gr \mathcal{M} \otimes p^{\ast}_{1} \omega_{X/Y} )[d]\langle -2d \rangle. $$
\end{proposition}
\begin{proof}
 For a mixed Hodge module $M=(\mathcal{M},F,\mathcal{K},W)$ on $Y$, $f^{\ast}[-d](M)$ is a mixed Hodge module which has underlying right $\mathcal{D}$-modules $\omega_{X/Y} \otimes_{ \mathcal{O}_{X}} f^{\ast} (\mathcal{M})$ and the filtration $\omega_{X/Y} \otimes_{\mathcal{O}_{X}} f^{\ast}F_{\bullet+ d}\mathcal{M} \cong  \omega_{X/Y} \otimes_{\mathcal{O}_{X}} f^{\ast}F_{\bullet}\mathcal{M}(-d)$, where the right $\mathcal{D}_{X}$-module structure is given by $\mathcal{D}_{X} \rightarrow f^{\ast} \mathcal{D}_{Y}$. See details in \cite[Section 30]{Schnell2014AnOO}.

 Hence by definition, $ gr(f^{\ast}\mathcal{M})[-d]\langle 2d \rangle = gr( f^{\ast}[-d](d) (\mathcal{M}))$ has filtration $\omega_{X/Y} \otimes_{\mathcal{O}_{X}} f^{\ast}F_{\bullet}\mathcal{M} $. After applying $\pi_{X\ast}$, the inverse image $f^{\ast}F_{\bullet}\mathcal{M}$ has associated graded module $ df^{{t}}_{\ast} p^{\ast}_{2}(gr \mathcal{M})$.  The proof is complete.
\end{proof}
 
The following proposition follows easily from the strictness of direct image, see details in \cite[Section 28]{Schnell2014AnOO}.
\begin{proposition} \label{pf}
	Assume $f$ is proper, then for any object $\mathcal{M}$ in  $\mathcal{D}^{b}({\rm{MHM}}(X,\mathbf{k}) )$ or $\mathcal{D}^{b}({\rm{MHM}}(X,\mathbf{k}) )'$, one has $gr(f_{\ast}\mathcal{M} ) \cong  p_{2\ast}(df^{t})^{\ast}(gr \mathcal{M} ) ). $
\end{proposition}

\subsection{Equivariant mixed Hodge modules and their associated graded moudules}
Following \cite{aequi} and \cite{Equivariant}, one can define the equivariant derived category of mixed Hodge modules. In this section, we introduce the associated graded functor of equivariant derived category of mixed Hodge modules and study the functorial property of this functor.

Let $G$ be a smooth affine algebraic group acting on a smooth complex variety. Recall that a $G$-resolution of $X$ is a smooth $G$-equivariant map $p: P \rightarrow X$ such that $P$ is a principal $G$-variety. A smooth morphism $\nu: (P \xrightarrow{p} X) \rightarrow (Q\xrightarrow{q} X)$ of resolutions is a smooth map $\nu: P \rightarrow Q$ such that $q\nu=p $.

An object $\mathcal{F}$ of the equivariant derived category $\mathcal{D}^{b}_{G}({\rm{MHM}}(X),\mathbf{k})$ of mixed Hodge modules on $X$ is a collection of objects $\mathcal{F}(p) \in  \mathcal{D}^{b}({\rm{MHM}}(P/G),\mathbf{k})$ for resolutions $p:P \rightarrow X$, and  isomorphisms $\alpha_{\nu}:\bar{\nu}^{\ast}\mathcal{F}(q) \rightarrow \mathcal{F}(p) $ for  smooth morphisms $\nu: (P \xrightarrow{p} X) \rightarrow (Q\xrightarrow{q} X)$ of resolutions, such that 
$\alpha_{id_{P}}=id_{\mathcal{F}(p)}$ and $\alpha_{\epsilon\circ \nu}=\alpha_{\nu} \circ \bar{\nu}^{\ast} \alpha_{\epsilon}$ for any composable smooth morphisms $P \xrightarrow{\nu} Q \xrightarrow{\epsilon} R$ of resolutions.

A morphism $\phi : \mathcal{F} \rightarrow \mathcal{G}$ in  the equivariant derived category $\mathcal{D}^{b}_{G}({\rm{MHM}}(X),\mathbf{k})$ is a collection of 
morphisms $\phi(p):\mathcal{F}(p) \rightarrow \mathcal{G}(p)$ for resolutions $p:P \rightarrow X$ such that  for any smooth morphism $\nu: P \rightarrow Q$ of resolutions the following diagram commutes
\[
\xymatrix{
	\mathcal{F}(p) \ar[d]_{\phi(p)} \ar[r]^{\alpha_{\nu}}
	& \mathcal{F}(q) \ar[d]^{\phi(q)} \\	
	\mathcal{G}(p) \ar[r]^{\alpha_{\nu}}
	& \mathcal{G}(q) .
}
\]
The category $\mathcal{D}^{b}_{G}({\rm{MHM}}(X),\mathbf{k})$ is triangulated and equipped with a t-structure, and the heart is the abelian category of equivariant mixed Hodge modules. One can see \cite[Theorem 8.3]{aequi} for details.

\subsubsection{Definition of $gr$ for equivariant categories}

We can define the functor $$gr: \mathcal{D}^{b}_{G}({\rm{MHM}}(X),\mathbf{k})' \longrightarrow  \mathcal{D}^{b}({\rm{Qcoh}}^{G \times \mathbb{C}^{\times}}(T^{\ast} X) )'$$ as the following.
 Consider the projection $pr_{2}:G\times X \rightarrow X$, where $G$ acts on $G\times X$ by $g\cdot(h,x)=(hg^{-1},g\cdot x)$, then the quotient map $G\times X \rightarrow G\times X/G$ can be identified with the action map $\sigma : G \times X \rightarrow X$. For any object $\mathcal{F}$ in $\mathcal{D}^{b}_{G}({\rm{MHM}}(X),\mathbf{k})$, $\mathcal{F} \mapsto \mathcal{F}(pr_{2}), \phi \mapsto \phi(pr_{2})$ defines a t-exact functor 
$${\rm{For}}:\mathcal{D}^{b}_{G}({\rm{MHM}}(X),\mathbf{k}) \longrightarrow  \mathcal{D}^{b}({\rm{MHM}}(X),\mathbf{k}). $$

Consider the resolution $pr_{3}:G \times G \times X \rightarrow X$ and the morphisms of resolutions given by projections $pr_{23},pr_{13}:G\times G \times  X \rightarrow G \times X $, where $G$ acts on $G \times G \times X $ by $g\cdot(h_{1},h_{2},x)=(h_{1}g^{-1},h_{2}g^{-1},g\cdot x)$.  Noticing that $\bar{pr_{23}}$ and $\bar{pr_{13}}$ can be identified with $pr_{2}$ and $\sigma$ respectively, we obtain isomorphisms $\alpha_{pr_{23}}: pr_{2}^{\ast} {\rm{For}}(\mathcal{F}) \rightarrow \mathcal{F}(pr_{3})$ and $\alpha_{pr_{13}}: \sigma^{\ast} {\rm{For}}(\mathcal{F}) \rightarrow \mathcal{F}(pr_{3})$. We denote $\theta=\alpha_{pr_{23}}^{-1} \alpha_{pr_{13}}:  \sigma^{\ast} {\rm{For}}(\mathcal{F}) \rightarrow  pr_{2}^{\ast} {\rm{For}}(\mathcal{F}).$ By \cite[Lemma 6.4.6]{achar2021perverse}, the isomorphism $\theta$ satisfies \begin{equation}\label{eq}
	b^{\ast} \theta \circ pr_{23}^{\ast} \theta=m^{\ast} \theta,
\end{equation}
where $$m:G \times G\times X \rightarrow G\times X, (g,h,x) \mapsto  (gh,x),$$
$$b:G \times G\times X \rightarrow G\times X, (g,h,x) \mapsto  (g,hx),$$
$$pr_{23}:G \times G\times X \rightarrow G\times X, (g,h,x) \mapsto  (h,x).$$

In particular, after applying the associated graded functor, one get an isomorphism
\begin{equation} \label{gr}
	gr(\theta): gr( \sigma^{\ast} {\rm{For}}(\mathcal{F}) ) \rightarrow  gr( pr_{2}^{\ast} {\rm{For}}(\mathcal{F})). 
\end{equation}

 Notice that the cotangent correspondence of $pr_{2}$ is given by 
\[
\xymatrix{
	T^{\ast}G \times T^{\ast}X \ar[d]_{\pi_{G} \times id}
	& G \times  T^{\ast}X \ar[l]_{dpr_{2}^{t}} \ar[r]^-{p} \ar[dl]_{=}\ar[d]^{\tilde{pr_{2}}}
	& G \times X \ar[d]^{pr_{2}} \\	
	G \times T^{\ast}X
	& T^{\ast}X \ar[r]^{ \pi_{X} }
	& X,
}
\]
where $\tilde{pr_{2}}$ is the projection associated to $T^{\ast} X$ and $p$ is the composition of $\tilde{pr}_{2}$ and $\pi_{X}$. Denote $\tilde{\sigma}:G \times T^{\ast} X \rightarrow T^{\ast} X$ the action map of $T^{\ast}X$, then the cotangent correspondence of $\sigma$ fits into the following diagram
\[
\xymatrix{
	T^{\ast}G \times T^{\ast}X \ar[dr]_{\pi_{G} \times id}
	& G \times X \times_{X}  T^{\ast}X \ar[l]_{d\sigma^{t}} \ar[r]^-{p_{1} } \ar[d]^{\rho}
	& G \times X \ar[d]^{id \times pr_{2}} \\	
	& G \times T^{\ast}X \ar[r]^{\tilde{p}_{1} }  \ar[d]^{\tilde{\sigma}}
	& G \times X \ar[d]^{\sigma} \\
   	& T^{\ast} X \ar[r]^{\pi_{X}} & X,
}
\]
where $\tilde{p}_{1}=id \times \pi_{X}$  and $\rho$ is the isomorphism sending an element $(g,x,\omega \in T^{\ast}_{gx}X)$ to $(g,d\sigma_{g}^{\ast}\omega \in T^{\ast}_{x}X)$. Apply $(\pi_{G} \times id)_ {\ast}$ to equation (\ref{gr}) and use Proposition \ref{pb}, one get an isomorphism 
$$\tilde{\sigma}^{\ast}gr({\rm{For}}(\mathcal{F}))= \rho_{\ast} \rho^{\ast}\tilde{\sigma}^{\ast}gr({\rm{For}}(\mathcal{F}))  \rightarrow \tilde{pr_{2}}^{\ast} gr({\rm{For}}(\mathcal{F})), $$
which is still denoted by $gr(\theta)$. The isomorphism $gr(\theta)$ satisfies the equation (\ref{eq}) for $T^{\ast}X$ by a similar argument as \cite[Lemma 6.4.6]{achar2021perverse}. Hence we define a functor 
$$ gr: \mathcal{D}^{b}_{G}({\rm{MHM}}(X),\mathbf{k})' \longrightarrow  \mathcal{D}^{b}({\rm{Qcoh}}^{G \times \mathbb{C}^{\times}}(T^{\ast} X) )', \quad \mathcal{M} \mapsto (gr({\rm{For} } (\mathcal{M})), gr(\theta) ). $$

Let $\Lambda$ be a $G \times \mathbb{C}^{\times}$-invariant closed subvariety of $T^{\ast}X$ and denote by $\mathcal{D}^{b}_{G}({\rm{MHM}}(X),\Lambda,\mathbf{k})$ the full subcategory of $\mathcal{D}^{b}_{G}({\rm{MHM}}(X),\mathbf{k})$ consisting of objects whose character varieties are contained in $\Lambda$, then $gr$ also restricts to a functor $$ gr: \mathcal{D}^{b}_{G}({\rm{MHM}}(X),\Lambda,\mathbf{k})' \longrightarrow  \mathcal{D}^{b}({\rm{Qcoh}}^{G \times \mathbb{C}^{\times}}_{\Lambda}(T^{\ast}X) )' \xrightarrow{\cong} \mathcal{D}^{b}({\rm{Qcoh}}^{G \times \mathbb{C}^{\times}}(\Lambda) )'. $$
The functor $gr$ can also induce a morphism $[gr]$ of $\mathbb{Z}[v^{\pm}]$-modules from the Grothendieck group of mixed Hodge modules to the equivariant $K$-group of $\Lambda$.

\subsubsection{Functorial property of $gr$ for equivariant categories}

Assume $f:X \rightarrow Y$ is a smooth $G$-equivariant morphism of relative dimension $d$, then the following results follow from Proposition \ref{pb} and \ref{pf}.
\begin{proposition} \label{epb}
	For any object $\mathcal{M}$ in  $\mathcal{D}^{b}_{G}({\rm{MHM}}(Y,\mathbf{k}) )'$, there is an isomorphism $$ gr(f^{\ast}\mathcal{M} ) \cong df^{{t}}_{\ast} p^{\ast}_{2}(gr \mathcal{M} \otimes p^{\ast}_{1} \omega_{X/Y} )[d]\langle -2d \rangle $$ in $\mathcal{D}^{b}({\rm{Qcoh}}^{G \times \mathbb{C}^{\times}}(T^{\ast}X) )'$.  Moreover, if $\Lambda,\Lambda'$ are close $G\times \mathbb{C}^{\times}$-invariant subvarieties of $T^{\ast}X$ and $T^{\ast}Y$ respectively and $df^{t} p_{2}^{-1}(\Lambda') \subseteq \Lambda$, then for any object $\mathcal{M}$ in $\mathcal{D}^{b}_{G}({\rm{MHM}}(Y,\Lambda',\mathbf{k}) )'$, there is an isomorphism $$gr(f^{\ast}\mathcal{M} ) \cong df^{{t}}_{\ast} p^{\ast}_{2}(gr \mathcal{M} \otimes p^{\ast}_{1} \omega_{X/Y} )[d]\langle -2d \rangle$$ in $\mathcal{D}^{b}({\rm{Qcoh}}^{G \times \mathbb{C}^{\times}}(\Lambda) )'$. Here $\omega_{X/Y}$ carries the natural $G$-equivariant structure. 
\end{proposition}

\begin{proposition} \label{epf}
	Assume $f$ is proper, then for any object $\mathcal{M}$ in  $\mathcal{D}^{b}_{G}({\rm{MHM}}(X,\mathbf{k}) )$ or $\mathcal{D}^{b}_{G}({\rm{MHM}}(X,\mathbf{k}) )'$, there is an isomorphism $gr(f_{\ast}\mathcal{M} ) \cong  p_{2\ast}(df^{t})^{\ast}(gr \mathcal{M} )) $ in $\mathcal{D}^{b}({\rm{Qcoh}}^{G \times \mathbb{C}^{\times}}(T^{\ast}Y) )'$. 
	
	Moreover, if $\Lambda,\Lambda'$ are closed $G\times \mathbb{C}^{\times}$-invariant subvarieties of $T^{\ast}X$ and $T^{\ast}Y$ respectively and $p_{2}(df^{t})^{-1}(\Lambda) \subseteq \Lambda'$, then for any object $\mathcal{M}$ in $\mathcal{D}^{b}_{G}({\rm{MHM}}(X,\Lambda,\mathbf{k}) )'$, there is an isomorphism $$gr(f_{\ast}\mathcal{M} ) \cong  p_{2\ast}(df^{t})^{\ast}(gr \mathcal{M}  ) $$ in $\mathcal{D}^{b}({\rm{Qcoh}}^{G \times \mathbb{C}^{\times}}(\Lambda') )'$.
\end{proposition}

Recall that if $H$ is a closed subgroup of $G$, then there is a forgetful functor 
$${\rm{For} }^{G}_{H}: \mathcal{D}^{b}_{G}({\rm{MHM}}(X,\mathbf{k}) ) \rightarrow \mathcal{D}^{b}_{H}({\rm{MHM}}(X,\mathbf{k}) )$$ such that ${\rm{For} }^{H}_{K}{\rm{For} }^{G}_{H} \cong {\rm{For} }^{G}_{K}$ for closed subgroup $K$ of $H$. Noticing that ${\rm{For} }^{G}_{H}$ is exactly the functor ${\rm{For} }$ when $H$ is trivial, we can show the following proposition.
\begin{proposition} \label{ef}
	For any object $\mathcal{M}$ in  $\mathcal{D}^{b}_{G}({\rm{MHM}}(X,\mathbf{k}) )$ or $\mathcal{D}^{b}_{G}({\rm{MHM}}(X,\mathbf{k}) )'$, there is an isomorphism $gr({\rm{For} }^{G}_{H}\mathcal{M} ) \cong  {\rm{For} }^{G}_{H}(gr \mathcal{M} )$ in $\mathcal{D}^{b}({\rm{Qcoh}}^{H \times \mathbb{C}^{\times}}(T^{\ast}X) )'$. Here the second ${\rm{For} }^{G}_{H}$ is the forgetful functor of quasi-coherent sheaves.
\end{proposition}

Assume $H$ is a close normal subgroup of $G$ and $X$ is a $G/H$-variety, then there is a inflation functor ${\rm{Infl}}^{G}_{G/H}: \mathcal{D}^{b}_{G/H}({\rm{MHM}}(X,\mathbf{k}) )\rightarrow \mathcal{D}^{b}_{G}({\rm{MHM}}(X,\mathbf{k}) )$ sending $\mathcal{M}$ to ${\rm{Infl}}^{G}_{G/H}\mathcal{M}(p)=\mathcal{M}(\bar{p})$, where $\bar{p}$ is the $G/H$-resolution $P/H \rightarrow X$ induced by the $G$-resolution $p: P \rightarrow X$. 

We assume $X$ is a $G$-variety and $H$ is a closed normal subgroup  of $G$ such that $X$ is a principal $H$-variety. Let $\pi: X \rightarrow X/H$ be the quotient map, then the quotient equivalence (equivariant descent) in \cite[Theorem 6.5.9]{achar2021perverse} is defined by 
$$ \pi^{\ast} {\rm{Infl}}^{G}_{G/H} : \mathcal{D}^{b}_{G/H}({\rm{MHM}}(X/H,\mathbf{k}) ) \rightarrow \mathcal{D}^{b}_{G}({\rm{MHM}}(X,\mathbf{k}) ). $$

With the notations above, one has an isomorphism $\mu_{H}^{-1}(0)/H \cong T^{\ast}(X/H)$. The inflation functor of quasi-coherent sheaves gives an equivalence $$\tilde{\pi}^{\ast} {\rm{Infl}}^{G}_{G/H}: \mathcal{D}^{b}({\rm{Qcoh}}^{G/H \times \mathbb{C}^{\times}}(T^{\ast}(X/H))  \rightarrow  \mathcal{D}^{b}({\rm{Qcoh}}^{G \times \mathbb{C}^{\times}}(  \mu_{H}^{-1}(0) ). $$ Here $\mu_{H}: T^{\ast}X \rightarrow Lie(H)$ is the moment map and $\tilde{\pi}:\mu_{H}^{-1}(0) \rightarrow T^{\ast}(X/H) $ is the quotient map.

\begin{proposition}\label{eqo}
	With the notations above, for any object $\mathcal{M}$ in $\mathcal{D}^{b}_{G/H}({\rm{MHM}}(X/H,\mathbf{k}) )$ or in $\mathcal{D}^{b}_{G/H}({\rm{MHM}}(X/H,\mathbf{k}) )'$,  one has the following isomorphism
	$$ gr ( \pi^{\ast} {\rm{Infl}}^{G}_{G/H}(\mathcal{M} )\cong \tilde{\pi}^{\ast} {\rm{Infl}}^{G}_{G/H} gr(\mathcal{M})\otimes p^{\ast}_{1}\omega_{X/(X/H)}[{\rm{dim}}(H)]\langle -2 {\rm{dim}}(H) \rangle  $$
	 in  $ \mathcal{D}^{b}({\rm{Qcoh}}^{G \times \mathbb{C}^{\times}}(  \mu_{H}^{-1}(0) )'$.
\end{proposition}
\begin{proof}
	Notice that the cotangent correspondence of $\pi$ is given by 
\[
\xymatrix{
	T^{\ast}X \ar[d]_{\pi_{X}}
	& \mu_{H}^{-1}(0) \ar[l]_{d\pi^{t}} \ar[r]^{p_{1}} \ar[dl]_{p_{1}} \ar[d]^{\tilde{\pi}}
	& X \ar[d]^{\pi} \\	
	X
	& T^{\ast}(X/H) \ar[r]^{\pi_{X/H}}
	& X/H.
}
\]
	By Proposition \ref{epb}, one has $$gr ( \pi^{\ast} {\rm{Infl}}^{G}_{G/H}(\mathcal{M} ))\cong  \tilde{\pi}^{\ast} gr ({\rm{Infl}}^{G}_{G/H}(\mathcal{M} ) )\otimes \omega_{X/(X/H)} [{\rm{dim}}(H)]\langle - 2 {\rm{dim}}(H) \rangle  .$$ By \cite[Lemma 6.5.8]{achar2021perverse}, ${\rm{For}}^{G}{\rm{Infl}}^{G}_{G/H} \cong {\rm{For}}^{G/H}$, hence the functor $gr$ commutes with inflation functors and the proof is finished.
\end{proof}

Assume $X$ is a $H$-variety and let $i :X \rightarrow G\times^{H} X$ be the natural embedding $x \mapsto (e,x)$, where $e$ is the identity of $G$. The induction equivalence $\mathcal{D}^{b}_{G}({\rm{MHM}}(G\times^{H} X,\mathbf{k}) ) \xrightarrow{\cong} \mathcal{D}^{b}_{H}({\rm{MHM}}(X,\mathbf{k}))$ is given by the composition $i^{\ast}[-{\rm{dim}} (G/H)] {\rm{For}}^{G}_{H}$, and we denote its inverse by ${\rm{Ind}}^{G}_{H}$. However, the map $i$ is not smooth in general, and we can not apply Proposition \ref{epb} directly, so we need to decompose ${\rm{Ind}}^{G}_{H}$ into some quotient equivalence. Following the proof of \cite[Theorem 6.5.10]{achar2021perverse}, we consider the following commutative diagram
\[
\xymatrix{ 
	X \ar[rr]^{i}&  & G\times^{H} X
 	\\
	H \times X \ar[u]^{\sigma} \ar[rr]_{\tilde{i}} \ar[dr]_{pr_{2}}
	& 
	& G \times X \ar[dl]^{pr_{2}} \ar[u]_{\pi} \\	
	& X	& 
}
\]
where $\pi$ is the quotient map, and $G \times H$ acts on $G \times X$ by $(g,h) \cdot (g',x)=(gg'h^{-1},h \cdot x) $. Then $i^{\ast} {\rm{For}}^{G}_{H}$ is the composition of the following equivalences (or their inverses) 
$$ pr_{2}^{\ast} {\rm{Infl}} ^{G \times H}_{1\times H}:  \mathcal{D}^{b}_{H}({\rm{MHM}}(X,\mathbf{k})) \rightarrow \mathcal{D}^{b}_{G \times H}({\rm{MHM}}(G \times X,\mathbf{k})) ,$$
$$ pr_{2}^{\ast} {\rm{Infl}} ^{H \times H}_{1\times H}:  \mathcal{D}^{b}_{H}({\rm{MHM}}(X,\mathbf{k})) \rightarrow \mathcal{D}^{b}_{H \times H}({\rm{MHM}}(H \times X,\mathbf{k})) ,$$
$$ \sigma^{\ast} {\rm{Infl}} ^{H \times H}_{H\times 1}:  \mathcal{D}^{b}_{H}({\rm{MHM}}(X,\mathbf{k})) \rightarrow \mathcal{D}^{b}_{H \times H}({\rm{MHM}}(H \times X,\mathbf{k})) ,$$
and 
$$ \pi^{\ast} {\rm{Infl}} ^{G \times H}_{G\times 1}:  \mathcal{D}^{b}_{G}({\rm{MHM}}(G \times ^{H}X,\mathbf{k})) \rightarrow \mathcal{D}^{b}_{G \times H}({\rm{MHM}}(G \times X,\mathbf{k})) .$$
Let $\Lambda \subseteq T^{\ast}X$ be a $H\times \mathbb{C}^{\times}$-invariant subvariety and denote $\Lambda'= G\times^{H} \Lambda \subseteq T^{\ast}(G\times^{H}X)$.

\begin{proposition}
	With the notation above, for any object $\mathcal{M}$ in  $\mathcal{D}^{b}_{H}({\rm{MHM}}(X,\Lambda,\mathbf{k}) )'$, there is an isomorphism 
	$$ gr ( {\rm{Ind}}^{G}_{H} (\mathcal{M})) \cong {\rm{Ind}}^{G}_{H}( gr (  \mathcal{M}) \otimes p^{\ast}_{1}\omega^{-1}_{X/G \times^{H} X})[{\rm{dim}} G/H ]\langle -2 {\rm{dim}} G/H  \rangle $$ 
	in $\mathcal{D}^{b}({\rm{Qcoh}}^{G \times \mathbb{C}^{\times}}(\Lambda') )'$. Here the second ${\rm{Ind}}^{G}_{H}$ is the induction equivalence $\mathcal{D}^{b}({\rm{Qcoh}}^{H \times \mathbb{C}^{\times}}(\Lambda) ) \xrightarrow{\cong} \mathcal{D}^{b}({\rm{Qcoh}}^{G \times \mathbb{C}^{\times}}(\Lambda') )$  of quasi-coherent sheaves, and $p_{1}:\Lambda \rightarrow X$ is the natural projection.
\end{proposition}

\begin{proof}
	The cotangent correspondence of 
\[
\xymatrix{ 
	H \times X  \ar[rr]_{\tilde{i}} \ar[dr]_{pr_{2}}
	& 
	& G \times X \ar[dl]^{pr_{2}}  \\	
	& X	& 
}
\]
is given by 
\[
\xymatrix{ 
	H \times T^{\ast}X  \ar[rr]_{\tilde{\iota}} \ar[dr]_{pr_{2}}
	& 
	& G \times T^{\ast}X \ar[dl]^{pr_{2}}  \\	
	& T^{\ast}X	&. 
}
\]
After restricting to $\Lambda$, one get $gr(\tilde{i}^{\ast} {\rm{Infl}}^{G\times H}_{H \times H} \mathcal{M}') \cong \tilde{\iota}^{\ast}  gr(\mathcal{M}')  \otimes p^{\ast}\omega_{H \times X/G \times X}[-{\rm{dim}} G/H ]\langle 2 {\rm{dim}} G/H  \rangle  $ by Proposition \ref{eqo}. Here $p: H \times T^{\ast}X \rightarrow H \times X$ is the natural projection, and we also denote its restriction on  $H \times \Lambda$ by $p$.

After restricting the cotangent correspondence of $\sigma : H \times X \rightarrow X$ to $\tilde{\sigma}: H \times \Lambda \rightarrow \Lambda$, the cotangent correspondences in the commutative diagram 
\[
\xymatrix{ 
	X \ar[r]^{i} &  G\times^{H} X
	\\
	H \times X \ar[u]^{\sigma} \ar[r]_{\tilde{i}} 
	& G \times X  \ar[u]_{\pi}
}
\]
is given by 
\[
\xymatrix{ 
	\Lambda \ar[r]^{\iota}&   G\times^{H} \Lambda
	\\
	H \times \Lambda \ar[u]^{\tilde{\sigma}} \ar[r]_{\tilde{\iota}} 
	& G \times \Lambda  \ar[u]_{\tilde{\pi}}.
}
\]
Use Proposition \ref{eqo} again, the proof is finished.
\end{proof}

\section{Mixed Hodge modules on quivers and their associated graded moudles}
In this section, we study the associated graded modules of mixed Hodge modules on the moduli space of quiver representations, and we always take $\mathbf{k}=\mathbb{Q} \subseteq \mathbb{R}$ to be the coefficients of underlying constructible sheaves. 
\subsection{Borcherds-Bozec algabra via mixed Hodge modules on quivers}\label{BB}
\subsubsection{Convolution}
Let $Q=(Q_{0},Q_{1},s,t)$ be a quiver in which loops and multiple arrows are allowed, and $\alpha \in \mathbb{N}Q_{0}$ be a dimension vector. The moduli space of representation of $Q$ with dimension $\alpha$ is the vector space 
$$ \mathbf{E}_{\alpha}=\mathbf{E}_{\alpha,Q} = \bigoplus_{e\in Q_{1}} {\rm{Hom}} (\mathbb{C}^{\alpha_{s(e)}} ,\mathbb{C}^{\alpha_{t(e)}}  ) $$
together with the action of $\mathbf{G}_{\alpha}=\prod_{i \in Q_{0}} GL(\mathbb{C}^{\alpha_{i}})$ given by
$ g \cdot x_{e} =  g_{t(e)} x_{e} g_{s(e)}^{-1}, $ where $x_{e}$ is an element of ${\rm{Hom}} (\mathbb{C}^{\alpha_{s(e)}} ,\mathbb{C}^{\alpha_{t(e)}}  )$.

Given dimension vector $\alpha,\beta$ and $\gamma =\alpha +\beta$, and fix a graded subspace $\mathbf{V} \subseteq \bigoplus_{i \in Q_{0}} \mathbb{C}^{\gamma_{i}}$ of dimension vector $\beta$, we consider the following diagram
$$ \mathbf{E}_{\alpha} \times \mathbf{E}_{\beta} \xleftarrow{p} \mathbf{F} \xrightarrow{ i} \mathbf{G}_{\gamma} \times^{\mathbf{P}} \mathbf{F} \xrightarrow{q} \mathbf{E}_{\gamma},  $$
where $\mathbf{F}$ is the subspace of $\mathbf{E}_{\gamma}$ consisting of elements stablizing the subspace $\mathbf{V}$, $\mathbf{P}$ is the parabolic subgroup of $\mathbf{G}_{\gamma}$ stablizing $\mathbf{V}$, the morphism $p$ is the natural projection, $i$ is the natural embedding and $q$ sends $(g,x)$ to $g \cdot x$.

Denote the equivariant derived category $\mathcal{D}^{b}_{\mathbf{G}_{\alpha}} ({\rm{MHM}}( \mathbf{E}_{\alpha},\mathbb{Q} ))'$ by   $\mathcal{D}_{\alpha}=\mathcal{D}_{\alpha,Q}$, the convolution $\star$ of $ \coprod_{\alpha} \mathcal{D}_{\alpha} $ is defined by 
$$ \mathcal{M} \star \mathcal{N}= q_{\ast} \circ {\rm{Ind}}^{\mathbf{G}_{\gamma}}_{\mathbf{P}}  \circ p^{\ast} {\rm{Infl}}^{\mathbf{P}}_{\mathbf{G}_{\alpha} \times \mathbf{G}_{\beta}}  (  \mathcal{M} \boxtimes \mathcal{N}) [\langle \langle  \alpha ,\beta \rangle \rangle] (\frac{1}{2 } \langle \langle  \alpha ,\beta \rangle \rangle), $$
where $\langle \langle  \alpha ,\beta \rangle \rangle =\sum\limits_{i \in Q_{0}} \alpha_{i} \beta_{i} + \sum\limits_{e \in Q_{1}} \alpha_{s(e)} \beta_{t(e)}$ is the geometric paring of the quiver $Q$.

\subsubsection{Borcherds-Bozec algabra}
Following \cite{Bo15}, take $m\in \mathbb{N}$ and two sequences $\mathbf{i}=(i_{1},i_{2}.\cdots,i_{m} )$ and $\mathbf{a}=(a_{1},a_{2},\cdots,a_{m})$ of $Q_{0}$ and $\mathbb{N}$ respectively.
If $\sum\limits_{k=1}^{m}a_{k}i_{k}=\alpha $, we write $\underline{\rm{dim}}(\mathbf{i},\mathbf{a})=\alpha$. Take a $Q_{0}$-graded space $\mathbf{V}$ of dimension vector $\alpha$, define varieties 
$$\mathcal{F}_{\mathbf{i},\mathbf{a}}=\{\mathbf{W}_{\bullet}=(0=\mathbf{W}_{0} \subseteq \mathbf{W}_{1} \subseteq \cdots \subseteq \mathbf{W}_{m}=\mathbf{V})| \text{ for any } k, {\rm{dim}} \mathbf{W}_{k}/\mathbf{W}_{k-1}=a_{k}i_{k}.   \}$$ 
$$ \tilde{\mathbf{E}}_{\mathbf{i},\mathbf{a}}=\{ (x,\mathbf{W}_{\bullet})| x\in \mathbf{E}_{\alpha}, \mathbf{W}_{\bullet} \in \mathcal{F}_{\mathbf{i},\mathbf{a}},\text{ and for any } k \geqslant 1, x(\mathbf{W}_{k}) \subseteq \mathbf{W}_{k-1}. \} $$

Consider the mixed Hodge module $\underline{\mathbb{Q}}^{H}_{\tilde{\mathbf{E}}_{\mathbf{i},\mathbf{a}}}$ on $\tilde{\mathbf{E}}_{\mathbf{i},\mathbf{a}}$ and denote it by $\mathbf{1}_{\tilde{\mathbf{E}}_{\mathbf{i},\mathbf{a}}}$, then $\mathbf{1}_{\tilde{\mathbf{E}}_{\mathbf{i},\mathbf{a}}}$ is a pure mixed Hodge module of weight $0$. Let $\pi_{\mathbf{i},\mathbf{a}}: \tilde{\mathbf{E}}_{\mathbf{i},\mathbf{a}} \rightarrow \mathbf{E}_{\alpha}, (x,\mathbf{w}_{\bullet}) \mapsto x$ be the projection, which is projective, then $(\pi_{\mathbf{i},\mathbf{a}})_{!} \mathbf{1}_{\tilde{\mathbf{E}}_{\mathbf{i},\mathbf{a}}}$ is a semisimple complex of weight $0$. In particular, if $\mathbf{i}=(i)$ and $\mathbf{a}=(a)$, we denote  $(\pi_{\mathbf{i},\mathbf{a}})_{!} \mathbf{1}_{\tilde{\mathbf{E}}_{\mathbf{i},\mathbf{a}}}$ by $\mathbf{1}_{ai}$, whose underlying constructible sheaf is the skyscraper sheaf supported at $0 \in \mathbf{E}_{ai}$. We remark that the author of \cite{Bo15} uses the constant perverse sheaves on $\mathbf{E}_{ai}$ as generators, which differ from ours by a Fourier transform.

Assuming that $(\pi_{\mathbf{i},\mathbf{a}})_{!} \mathbf{1}_{\tilde{\mathbf{E}}_{\mathbf{i},\mathbf{a}}}=\bigoplus\limits_{j \in J} IC(\bar{Y}_{j},\mathbb{Q})[k_{j}], $ then we must have $IC(\bar{Y}_{j},\mathbb{Q})$ is pure of weight $-k_{j}$ for every $j \in J$, and we call such $IC(\bar{Y}_{j},\mathbb{Q})[k_{j}]$ a weight $0$ shifted simple direct summand of  $(\pi_{\mathbf{i},\mathbf{a}})_{!} \mathbf{1}_{\tilde{\mathbf{E}}_{\mathbf{i},\mathbf{a}}}$

Let $\mathcal{P}_{\alpha}=\mathcal{P}_{\alpha,Q}$ be the set of weight $0$ shifted simple direct summands of those  $(\pi_{\mathbf{i},\mathbf{a}})_{!} \mathbf{1}_{\tilde{\mathbf{E}}_{\mathbf{i},\mathbf{a}}}$ such that $   \underline{\rm{dim}}(\mathbf{i},\mathbf{a})=\alpha,$ and let $\mathcal{Q}_{\alpha}=\mathcal{Q}_{\alpha,Q}$ be the full subcategory of $\mathcal{D}_{\alpha}$ consisting of complexes of the form $$\bigoplus\limits_{L \in \mathcal{P}_{\alpha},k \in \mathbb{Z}}  L^{\oplus m_{L,k} }[k](\frac{k}{2})$$ such that $m_{L,k} \in \mathbb{N}$ and only finitely many $m_{L,k}$ are nonzero.

Let $\mathcal{K}_{\oplus}(\mathcal{Q}_{\alpha})$ be the split Grothendieck group of $\mathcal{Q}_{\alpha}$. More precisely, $\mathcal{K}_{\oplus}(\mathcal{Q}_{\alpha})$ is the $\mathbb{Z}[v^{\pm}]$-module spanned by $[L],L \in \mathcal{Q}_{\alpha}$ and subject to the relations
$$ [L_{1} \oplus L_{2}]=[L_{1}] +[L_{2}], $$
$$  [L[k](\frac{k}{2})  ]=v^{k} [L], k\in \mathbb{Z}. $$

By a similar argument as \cite[Theorem 2.29]{Bo15}, the direct sum $\mathcal{K}_{\oplus}(\mathcal{Q})=\bigoplus\limits_{\alpha \in \mathbb{N}Q_{0}} \mathcal{K}_{\oplus}(\mathcal{Q}_{\alpha})$ with $\star$ becomes an associative algebra, and is isomorphic to $\mathbf{U}^{\mathbb{Z}}_{v}(\mathfrak{n}^{+}_{Q})$ via the canonical isomorphism determined by $[\mathbf{1}_{ai}] \mapsto E^{(a)}_{i}$ if $i \in Q_{0}^{re}$ and $[\mathbf{1}_{ai}] \mapsto E_{i,a}$ if $i \in Q_{0}^{im}$. This endows $\mathcal{K}_{0}(\mathcal{D})= \bigoplus \limits_{\alpha \in \mathbb{N}Q_{0}} \mathcal{K}_{0}(\mathcal{D}_{\alpha})$ with a  structure of $\mathbb{Z}[v^{\pm}]$-algebra. 

\begin{remark}
	Since we choose different generators, when using \cite[Theorem 2.29]{Bo15}, we also need the fact that Fourier transforms commute with the convolution. A proof for constructible sheaves can be found in \cite[10.4.5]{achar2021perverse}, and this proof also holds for mixed Hodge modules.
\end{remark}

\subsubsection{The $\Psi$-twist}

Let $\mathcal{K}'_{\oplus}(\mathcal{Q}_{\alpha})$ be the $\mathbb{Z}[v^{\pm}]$-module spanned by $[L],L \in \mathcal{Q}_{\alpha}$ and subject to the relations
$$ [L_{1} \oplus L_{2}]=[L_{1}] +[L_{2}], $$
$$  [L[k](\frac{k}{2})  ]=(-v)^{k} [L], k\in \mathbb{Z}, $$
then $(\mathcal{K}'_{\oplus}(\mathcal{Q})=\bigoplus\limits_{\alpha \in \mathbb{N}Q_{0}} \mathcal{K}'_{\oplus}(\mathcal{Q}_{\alpha}),\star  )$ is isomorphic to $\mathbf{U}^{\mathbb{Z}}_{-v}(\mathfrak{n}^{+}_{Q}).$

Take $\Psi: \mathbb{N}Q_{0} \times \mathbb{N}Q_{0} \rightarrow \{\pm 1 \}$, $\Psi(a,b)=(-1)^{ \langle \langle a,b \rangle \rangle  } $ , then 
one can define a twisted product $\star^{\Psi}: \mathcal{K}'_{\oplus}(\mathcal{Q}_{\alpha}) \times \mathcal{K}'_{\oplus}(\mathcal{Q}_{\beta}) \rightarrow \mathcal{K}'_{\oplus}(\mathcal{Q}_{\gamma}) $ or $ \star^{\Psi}: \mathcal{K}_{0}(\mathcal{D}_{\alpha}) \times \mathcal{K}_{0}(\mathcal{D}_{\beta}) \rightarrow \mathcal{K}_{0}(\mathcal{D}_{\gamma})$ by setting
$$ [\mathcal{M}] \star^{\Psi} [\mathcal{N}] = \Psi(\alpha,\beta)  [\mathcal{M}] \star [\mathcal{N}].  $$ By \cite[Proposition 8.7]{HLCC}, the split Grothencieck group $\mathcal{K}'_{\oplus}(\mathcal{Q})=\bigoplus\limits_{\alpha \in \mathbb{N}Q_{0}}\mathcal{K}'_{\oplus}(\mathcal{Q}_{\alpha})$ with $\star^{\Psi}$  is isomorphic to $\mathbf{U}^{\mathbb{Z}}_{v}(\mathfrak{n}^{+}_{Q})$ via the canonical isomorphism determined by $[\mathbf{1}_{ai}] \mapsto E^{(a)}_{i}$ if $i \in Q_{0}^{re}$ and $[\mathbf{1}_{ai}] \mapsto E_{i,a}$ if $i \in Q_{0}^{im}$.

\subsection{Twisted K-theoretic Hall algebra of preprojective algebra}\label{KHA}

For a given quiver $Q=(Q_{0},Q_{1},s,t)$, its double quiver is defined by $\overline{Q}=(Q_{0},\overline{Q}_{1},s,t)$ such that $\overline{Q}_{1}=Q_{1} \cup Q^{\ast}_{1}, Q^{\ast}_{1}=\{ e^{\ast}| e\in Q_{1} \}$ and $s(e^{\ast})=t(e), t(e^{\ast})=s(e)  $ for any $e \in Q_{1}$.

For any dimension vector $\alpha \in \mathbb{N}Q_{0}$, let $\overline{\mathbf{E}}_{\alpha}$ be the moduli space of double quiver representations
$$ \overline{\mathbf{E}}_{\alpha}=\overline{\mathbf{E}}_{\alpha,Q}= \bigoplus\limits_{e \in Q_{1}} {\rm{Hom}} (\mathbb{C}^{\alpha_{s(e)}} ,\mathbb{C}^{\alpha_{t(e)}}  ) \oplus \bigoplus\limits_{e^{\ast} \in Q^{\ast}_{1}} {\rm{Hom}} (\mathbb{C}^{\alpha_{s(e^{\ast})}} ,\mathbb{C}^{\alpha_{t(e^{\ast})}}  ),  $$
which can be identified with the cotangent bundle of $\mathbf{E}_{\alpha}$.  We write an element of $\overline{\mathbf{E}}_{\alpha}$ as $(x,\bar{x})$, where $x=(x_{e})_{e\in Q_{1}} \in \mathbf{E}_{\alpha}$ and $\bar{x}=(\bar{x}_{e^{\ast}})_{e\in Q_{1}} \in \bigoplus\limits_{e^{\ast} \in {Q}^{\ast}_{1}} {\rm{Hom}} (\mathbb{C}^{\alpha_{s(e^{\ast})}} ,\mathbb{C}^{\alpha_{t(e^{\ast})}}  ). $ 

The algebraic group $\mathbf{G}_{\alpha}$ acts on $\overline{\mathbf{E}}_{\alpha}$ by composition $g \cdot x_{e} =  g_{t(e)} x_{e} g^{-1}_{s(e)}$ and
$ g \cdot \bar{x}_{e^{\ast}} =  g_{s(e)} \bar{x}_{e^{\ast}} g^{-1}_{t(e)}.  $  Then a $\mathbf{G}_{\alpha}$ orbit of $\overline{\mathbf{E}}_{\alpha}$ corresponds to an isomorphic class of representations of the double quiver.

The multiplicative group $\mathbb{C}^{\times}$ acts on $\overline{\mathbf{E}}_{\alpha}$ by $t \cdot x_{e}=x_{e}$ and $t \cdot \bar{x}_{e^{\ast}}=t^{-1}\bar{x}_{e^{\ast}}$, which coincides with the $\mathbb{C}^{\times}$-action on cotangent bundles in Section \ref{grdef}. 

The moment map of $\overline{\mathbf{E}}_{\alpha}$ is given by
$$  \mu_{\alpha}: \overline{\mathbf{E}}_{\alpha} \longrightarrow Lie(\mathbf{G}_{\alpha})=\bigoplus\limits_{i \in Q_{0}} \mathfrak{gl}(\alpha_{i}),  (x,\bar{x}) \mapsto \sum\limits_{e \in Q_{1}} x_{e}\bar{x}_{e^{\ast}}-\bar{x}_{e^{\ast}}x_{e}, $$
then its zero fiber $\mu^{-1}_{\alpha}(0)$ is the space of representations of the preprojective algebra $\Pi_{Q}$ of $Q$.

Denote the quotient stack $\mu_{\alpha}^{-1}(0)/\mathbf{G}_{\alpha}$ by $\mathfrak{M}_{\alpha}=\mathfrak{M}_{\alpha,Q}$ and denote the equivariant K-theory $K_{\mathbb{C}^{\times}}(\mathfrak{M}_{\alpha} )= K_{\mathbf{G}_{\alpha} \times \mathbb{C}^{\times} }( \mu^{-1}_{\alpha}(0))$ by $\mathcal{A}_{\alpha}$, then the direct sum $\mathcal{A}=\bigoplus\limits_{\alpha \in \mathbb{N}Q_{0}} \mathcal{A}_{\alpha} $ has a structure of K-theoretic Hall algebra. To define the Hall product, we need to introduce the following diagram

\begin{center}
\[	\xymatrix{ 	
	&  \mathfrak{M}_{\alpha,\beta} \ar[dr]^{\tilde{q}} \ar[dl]_{\tilde{p}} &
	\\
	\mathfrak{M}_{\alpha} \times \mathfrak{M}_{\beta} 
	&  & \mathfrak{M}_{\gamma}.}
	\]
\end{center}
Here $\gamma=\alpha+\beta$, $ \mathfrak{M}_{\alpha}, \mathfrak{M}_{\beta}$ and $\mathfrak{M}_{\gamma}$ parameterize double quiver representations $\mathbf{V}_{\alpha},\mathbf{V}_{\beta}$ and $\mathbf{V}_{\gamma}$ of dimension vectors $\alpha,\beta,\gamma$ respectively, and the stack $\mathfrak{M}_{\alpha,\beta}$ parameterizes short exact sequences $0 \rightarrow \mathbf{V}_{\beta} \rightarrow \mathbf{V}_{\gamma} \rightarrow \mathbf{V}_{\alpha} \rightarrow 0 $ of double quiver representations. The map $\tilde{p}$ records the beginning and ending of the short exact sequence, while $\tilde{q}$ records the middle  of the short exact sequence.

For $i \in Q_{0}$, let $\mathcal{V}_{\alpha_{i}}$ be the tautological vector bundle on $\mathfrak{M}_{\alpha}$, which parametrizes the subspace $\mathbf{V}_{\alpha_{i}}$ at vertex $i$ of $\mathbf{V}_{\alpha}$. One can define $\mathcal{V}_{\beta_{i}}$ in a similar way. Then the K-theoretic Hall product is defined by 
$$\ast: \mathcal{A}_{\alpha} \otimes \mathcal{A}_{\beta}  \longrightarrow \mathcal{A}_{\gamma},$$
\begin{equation}\label{defKHA}
	\mathcal{M}_{1} \otimes \mathcal{M}_{2} \mapsto  \tilde{q}_{\ast} (sdet [ \sum\limits_{i \in Q_{0}} \frac{\mathcal{V}_{\beta_{i}}}{v \mathcal{V}_{\alpha_{i}}}- \sum\limits_{e \in Q_{1}} \frac{\mathcal{V}_{\beta_{t(e)}}}{v^{-1
		}\mathcal{V}_{\alpha_{s(e)}}}] \cdot \tilde{p}^{!} (\mathcal{M}_{1} \boxtimes \mathcal{M}_{2} )   ).
\end{equation}
Here the pullback $\tilde{p}^{!}$ is defined in \cite[2.1 and 4.1]{YZ}. For the notations of twist, one can see details in \cite[1.11]{ne25}, but here we use a different twist. The operator $\ast$ makes $(\mathcal{A},\ast)$ be an associative $\mathbb{Z}[v^{\pm}]=K_{\mathbb{C}^{\times}}(pt)$-algebra. Here $v$ is provided by the square root $\langle 1 \rangle$ considered in Section \ref{grt}.

The strongly seminilpotent variety $\Lambda_{\alpha}=\Lambda_{\alpha,Q} \subseteq \mu^{-1}_{\alpha}(0)$ is the closed subvariety, whose close points correspond to representations $M$ of preprojective algebras, which allow a filtration $0=M_{0} \subseteq M_{1}\subseteq \cdots \subseteq M_{r}= M$ such that $x_{e}M_{s} \subseteq M_{s-1} $ and $\bar{x}_{e^{\ast}}M_{s} \subseteq M_{s} $ for any $e \in Q_{1}$ and $1\leqslant s \leqslant r$. By \cite{BSV20}, it is a Lagrangian subvariety of $\overline{\mathbf{E}}_{\alpha}$. We denote the quotient stack $\Lambda_{\alpha}/\mathbf{G}_{\alpha}$ by $\mathfrak{N}_{\alpha}=\mathfrak{N}_{\alpha,Q}$ and denote the equivariant K-theory $K_{\mathbb{C}^{\times}}(\mathfrak{N}_{\alpha})= K_{\mathbb{C}^{\times}\times \mathbf{G}_{\alpha}}(\Lambda_{\alpha})  $ by $\mathcal{A}^{nil}_{\alpha}$. Then the restriction of $\ast$ to $\mathcal{A}^{nil}=\bigoplus\limits_{\alpha \in \mathbb{N}Q_{0}}\mathcal{A}^{nil}_\alpha$ makes $\mathcal{A}^{nil}$ be an associative algebra, and we still denote the restriction of the Hall product by $\ast$.

\subsection{Compare different convolutions}
Consider the full subcategory $\mathcal{D}^{b}_{\mathbf{G}_{\alpha}}({\rm{MHM}} (\mathbf{E}_{\alpha},\Lambda_{\alpha},\mathbf{k} ))'$ of $\mathcal{D}_{\alpha}$ and denote it by $\mathcal{D}^{nil}_{\alpha}$, then there is a morphism of $\mathbb{Z}[v^{\pm}]$-modules
$ [gr]: \mathcal{K}_{0}(\mathcal{D}^{nil}_{\alpha}) \longrightarrow \mathcal{A}^{nil}_{\alpha}.$ 

\begin{proposition}
	Let $\mathcal{K}_{0}(\mathcal{D}^{nil})= \bigoplus \limits_{\alpha \in \mathbb{N}Q_{0}}  \mathcal{K}_{0}(\mathcal{D}^{nil}_{\alpha})$,
	 and then the map $$[gr]:  (\mathcal{K}_{0}(\mathcal{D}^{nil}),\star^{\Psi}) \longrightarrow (\mathcal{A}^{nil},\ast)$$ is a morphism of $\mathbb{Z}[v^{\pm}]$-algebras. 
\end{proposition}
\begin{proof}
	Consider the following commutative diagram
	\[
	\xymatrix{ 
	\mathbf{G}_{\gamma}\times^{\mathbf{P}} (\mathbf{E}_{\alpha} \times \mathbf{E}_{\beta})	&  \mathbf{G}_{\gamma}\times^{\mathbf{P}} \mathbf{F} \ar[r]^{q} \ar[l]_-{p'} & \mathbf{E}_{\gamma}
		\\
		\mathbf{E}_{\alpha} \times \mathbf{E}_{\beta} \ar[u]^{i'}
		&  \mathbf{F} \ar[u]^{i}  \ar[l]_{p} & ,
	}
	\]
	 and then the $\star$ product can be reformulated as 
	$$ \mathcal{M} \star \mathcal{N}= q_{\ast} \circ   (p')^{\ast}{\rm{Ind}}^{\mathbf{G}_{\gamma}}_{\mathbf{P}}  \circ {\rm{Infl}}^{\mathbf{P}}_{\mathbf{G}_{\alpha} \times \mathbf{G}_{\beta}}  (  \mathcal{M} \boxtimes \mathcal{N}) [\langle \langle  \alpha ,\beta \rangle \rangle] (\frac{1}{2 } \langle \langle  \alpha ,\beta \rangle \rangle), $$
	for $\mathcal{M}$ and $\mathcal{N}$ in $\mathcal{D}_{\alpha}$ and $\mathcal{D}_{\beta}$ respectively.

	Denote $	\mathbf{G}_{\gamma}\times^{\mathbf{P}} (\mathbf{E}_{\alpha} \times \mathbf{E}_{\beta}) $, $ 	\mathbf{G}_{\gamma}\times^{\mathbf{P}} \mathbf{F}$ and $\mathbf{E}_{\gamma}$ by $X,X'$ and $Y$ respectively, then the cotangent correspondence of $p'$ and $q$ induces the following commutative diagram
	\[
	\xymatrix{
	T^{\ast}_{Y}(X\times X') \ar[r]^{\eta'} \ar[d]_{\kappa'} & T^{\ast}X'\times_{X'} Y \ar[r]^-{\pi_{X'}} \ar[d]^{dq^{t}} & T^{\ast}X' \\
	T^{\ast}X\times_{X} Y \ar[r]^{(dp')^{t}} \ar[d]_{\pi_{X}} & T^{\ast}Y & \\
T^{\ast}X & &
    }
	\]
	such that the square is Cartesian. There are three projections $pr_{1},pr_{2},pr_{3}: T^{\ast}_{Y}(X\times X') \rightarrow T^{\ast}X,T^{\ast}X',Y $ respectively, then $\kappa'=(pr_{1},pr_{3})$ and $\eta'=(-pr_{2},pr_{3})$, where $-pr_{2}$ is the composition of $pr_{2}$ and the automorphism of $T^{\ast}X$ acting by $-1$ on fibers .
	
	By \cite[Lemma 9.4]{HLCC}, the variety $T^{\ast}_{Y}(X\times X')$ is isomorphic to $\mathbf{G}_{\gamma}\times^{\mathbf{P}}\overline{\mathbf{F}}$. Here $\overline{\mathbf{F}}$ is the closed subvariety of $\overline{\mathbf{E}}_{\gamma}$ consisting of elements $(x,\bar{x})$ preserving the fixed subspace $\mathbf{V}$. Hence the cotangent correspondences of $i',p'$ and $q$ fit into the following commutative diagram
	\[
	\xymatrix{
\overline{\mathbf{E}}_{\alpha} \times \overline{\mathbf{E}}_{\beta} \ar[r]^-{\iota}	& T^{\ast}X & \mathbf{G}_{\gamma}\times^{\mathbf{P}} \overline{\mathbf{F}} \ar[l]_-{\kappa} \ar[r]^-{\eta}  &  T^{\ast}X' \\
\Lambda_{\alpha} \times \Lambda_{\beta} \ar[u] \ar[r]^-{\iota}	& \mathbf{G}_{\gamma}\times^{\mathbf{P}}(\Lambda_{\alpha} \times \Lambda_{\beta})  \ar[u] & \mathbf{G}_{\gamma}\times^{\mathbf{P}} \Lambda_{\alpha,\beta} \ar[u] \ar[r]^-{\eta} \ar[l]_-{\kappa} & \Lambda_{\gamma}, \ar[u]
	}
	\]
	here $\Lambda_{\alpha,\beta}$ is $\Lambda_{\gamma} \cap \overline{\mathbf{F}}$, all vertical arrows are inclusion, $\iota$ is the inclusion defining induction equivalence, $\kappa$ and $\eta$ are (restrictions of) the compositions of vertical and horizontal maps in the diagram of the cotangent correspondence, and the right square is Cartesian. Then $\iota$,$\kappa$ and $-\eta$ induce morphisms $\tilde{p}$ and $\tilde{q}$ between stacks, where $-\eta$ is the composition of $(pr_{2},pr_{3})$ and $\pi_{X}'$.  (The fact $\eta\kappa^{-1}( \mathbf{G}_{\gamma}\times^{\mathbf{P}}(\Lambda_{\alpha} \times \Lambda_{\beta}) ) \subseteq \Lambda_{\gamma} $ also implies that $\mathcal{K}_{0}(\mathcal{D}^{nil})$ is a subalgebra of $\mathcal{K}_{0}(\mathcal{D})$.)
	
	Thus we have the following equation
	\begin{equation*}
		\begin{split}
			&[gr]([\mathcal{M}]\star^{\Psi} [\mathcal{N}] )\\
			\cong& (-1)^{\langle \langle \alpha,\beta \rangle \rangle}  [gr(\mathcal{M} \star \mathcal{N})  ] \\
			\cong&(-1)^{\langle \langle \alpha,\beta \rangle \rangle} [gr(   q_{\ast} \circ   (p')^{\ast}{\rm{Ind}}^{\mathbf{G}_{\gamma}}_{\mathbf{P}}  \circ {\rm{Infl}}^{\mathbf{P}}_{\mathbf{G}_{\alpha} \times \mathbf{G}_{\beta}}  (  \mathcal{M} \boxtimes \mathcal{N}) [\langle \langle  \alpha ,\beta \rangle \rangle] (\frac{1}{2 } \langle \langle  \alpha ,\beta \rangle \rangle) ) ] \\
			\cong & v^{  \langle \langle \alpha,\beta \rangle \rangle}[gr(   q_{\ast} \circ   (p')^{\ast}{\rm{Ind}}^{\mathbf{G}_{\gamma}}_{\mathbf{P}}  \circ {\rm{Infl}}^{\mathbf{P}}_{\mathbf{G}_{\alpha} \times \mathbf{G}_{\beta}}  (  \mathcal{M} \boxtimes \mathcal{N}))] \\
			\cong &   v^{  \langle \langle \alpha,\beta \rangle \rangle}[  (\pi_{X'})_{\ast}(dq^t)^{\ast}  gr(   (p')^{\ast}{\rm{Ind}}^{\mathbf{G}_{\gamma}}_{\mathbf{P}}  \circ {\rm{Infl}}^{\mathbf{P}}_{\mathbf{G}_{\alpha} \times \mathbf{G}_{\beta}}  (  \mathcal{M} \boxtimes \mathcal{N}))]\\
			\cong &   (-1)^{d_{1}}v^{  \langle \langle \alpha,\beta \rangle \rangle-2d_{1} }[  (\pi_{X'})_{\ast}(dq^t)^{\ast}((dp')^{t})^{\ast}(\pi_{X})_{\ast} \\
			&  gr(   {\rm{Ind}}^{\mathbf{G}_{\gamma}}_{\mathbf{P}}  \circ {\rm{Infl}}^{\mathbf{P}}_{\mathbf{G}_{\alpha} \times \mathbf{G}_{\beta}}  (  \mathcal{M} \boxtimes \mathcal{N}) \otimes p^{\ast}_{1}\omega_{Y/X}) ] \\
			\cong &   (-1)^{d_{1}+d_{2}}v^{  \langle \langle \alpha,\beta \rangle \rangle-2d_{1}-2d_{2} }[  (\pi_{X'})_{\ast}(dq^t)^{\ast}((dp')^{t})^{\ast}(\pi_{X})_{\ast} {\rm{Ind}}^{\mathbf{G}_{\gamma}}_{\mathbf{P}} \\
			& gr(  {\rm{Infl}}^{\mathbf{P}}_{\mathbf{G}_{\alpha} \times \mathbf{G}_{\beta}}  (  \mathcal{M} \boxtimes \mathcal{N}) \otimes p^{\ast}_{2} \omega^{-1}_{\mathbf{E}_{\alpha} \times \mathbf{E}_{\beta}/X}  \otimes ({\rm{Ind}}^{\mathbf{G}_{\gamma}}_{\mathbf{P}})^{-1}  p^{\ast}_{1}\omega_{Y/X}) ] \\
			\cong & (-1)^{d_{1}+d_{2}}v^{  \langle \langle \alpha,\beta \rangle \rangle-2d_{1}-2d_{2} }[  (\eta)_{\ast}(\kappa)^{\ast} {\rm{Ind}}^{\mathbf{G}_{\gamma}}_{\mathbf{P}} {\rm{Infl}}^{\mathbf{P}}_{\mathbf{G}_{\alpha} \times \mathbf{G}_{\beta}}  \\
			& gr(  (  \mathcal{M} \boxtimes \mathcal{N}) \otimes p^{\ast}_{2}  ({\rm{Infl}}^{\mathbf{P}}_{\mathbf{G}_{\alpha} \times \mathbf{G}_{\beta}} )^{-1} \omega^{-1}_{\mathbf{E}_{\alpha} \times \mathbf{E}_{\beta}/X}  \otimes({\rm{Infl}}^{\mathbf{P}}_{\mathbf{G}_{\alpha} \times \mathbf{G}_{\beta}}  )^{-1} ({\rm{Ind}}^{\mathbf{G}_{\gamma}}_{\mathbf{P}})^{-1}  p^{\ast}_{1}\omega_{Y/X}) ] .
		\end{split}
	\end{equation*}
	Here $d_{1}$ is the relative dimension of $p'$, $d_{2}$ is the dimension of $\mathbf{G}_{\gamma}/\mathbf{P}$, $p_{1}$ is (the restriction of) $Y \times X T^{\ast}X  \rightarrow Y$ and $p_{2}$ is (the restriction of) $\overline{\mathbf{E}}_{\alpha} \times \overline{\mathbf{E}}_{\beta} \rightarrow \mathbf{E}_{\alpha} \times \mathbf{E}_{\beta}. $
	
	Noticing that the composition $(-\eta)_{\ast}(\kappa)^{\ast} {\rm{Ind}}^{\mathbf{G}_{\gamma}}_{\mathbf{P}} {\rm{Infl}}^{\mathbf{P}}_{\mathbf{G}_{\alpha} \times \mathbf{G}_{\beta}}$ defines $\tilde{q}_{\ast}\tilde{p}^{!}$ in the sense of \cite{YZ}, it remains to calculate the relative canonical sheaves of $i'$ and $p'$. (Here we also use the fact that  the actions of $(-\eta)_{\ast}$ and $\eta_{\ast}$ coincide on the K groups.)
	
	For $i \in Q_{0}$, we still denote the tautological vector bundle  on $\mathbf{E}_{\alpha}$  which associates with the vector space at $i$ by $\mathcal{V}_{\alpha_{i}}$, as what we have done for the preprojective representations. Since $p'$ is a trivial vector bundle whose fiber is  isomorphic to $\bigoplus\limits_{e \in Q_{1}} {\rm{Hom}} (\mathbb{C}^{\alpha_{s(e)}} ,\mathbb{C}^{\beta_{t(e)}} )$, $\omega_{Y/X}$ equals to $det(-\sum\limits_{e \in Q_{1}} \frac{\mathcal{V}_{\beta_{{t(e)}}}}{\mathcal{V}_{\alpha_{s(e)}}})$ in K-theory, and $({\rm{Infl}}^{\mathbf{P}}_{\mathbf{G}_{\alpha} \times \mathbf{G}_{\beta}}  )^{-1} ({\rm{Ind}}^{\mathbf{G}_{\gamma}}_{\mathbf{P}})^{-1}  p^{\ast}_{1}\omega_{Y/X}$ is the same determinant of the tautological vector bundle on $\Lambda_{\alpha} \times \Lambda_{\beta}$. 
	
	Since $\omega_{\mathbf{E}_{\alpha} \times \mathbf{E}_{\beta}/X} $ is determined by the adjoint representation $\mathfrak{g}_{\gamma}/\mathfrak{p}$ of $\mathbf{P}$, where $\mathfrak{g}_{\gamma}$ and $\mathfrak{p}$ are Lie algebras of $\mathbf{G}_{\gamma}$ and $\mathbf{P}$, a direct calculation shows that $\omega^{-1}_{\mathbf{E}_{\alpha} \times \mathbf{E}_{\beta}/X} =det (\sum\limits_{i \in Q_{0}}  \frac{\mathcal{V}_{\beta_{i}}}{\mathcal{V}_{\alpha_{i}}} ) $,  and $p^{\ast}_{2}  ({\rm{Infl}}^{\mathbf{P}}_{\mathbf{G}_{\alpha} \times \mathbf{G}_{\beta}} )^{-1}$ $\omega^{-1}_{\mathbf{E}_{\alpha} \times \mathbf{E}_{\beta}/X} $ is the same line bundle on  $\Lambda_{\alpha} \times \Lambda_{\beta}$.
	
	Noticing that $d_{1}+d_{2}=\langle \langle \alpha,\beta \rangle \rangle$, a direct calculation shows that the signs and $v$ powers in the twists coincide, the proof is finished. 
\end{proof}

By a similar argument as \cite[Section 13]{Lu91} and \cite{BT16}, one can prove that singular support of an object in $\mathcal{Q}_{\alpha}$ is contained in $\Lambda_{\alpha}$.  Noticing that there is a natural $\mathbb{Z}[^{\pm}]$-linear morphism $\mathcal{K}'_{\oplus}(\mathcal{Q}_{\alpha}) \rightarrow \mathcal{K}_{0}(\mathcal{D}_{\alpha}) $ of algebras, we still denote the composition  $\mathcal{K}'_{\oplus}(\mathcal{Q}_{\alpha}) \rightarrow \mathcal{K}_{0}(\mathcal{D}_{\alpha})\xrightarrow{[gr]}\mathcal{A}^{nil}_{\alpha}  $ by $[gr]$. 

\begin{proposition}\label{inj}
	The map $[gr]: (\mathcal{K}'_{\oplus}(\mathcal{Q}),\star^{\Psi}) \longrightarrow (\mathcal{A}^{nil},\ast)  $ is an injective homomorphism of algebras.
\end{proposition}

In order to show the injectivity, we need to introduce the (zero degree) cohomological Hall algebra.

Let $\mathcal{H}_{\alpha}={\rm{H}}^{\ast}(\mathfrak{N}_{\alpha},\mathbf{k})={\rm{H}}^{\ast,\mathbf{G}_{\alpha}}(\Lambda_{\alpha},\mathbf{k}) $ be the cohomology groups of $\mathfrak{N}_{\alpha}$, and let $\widehat{\mathcal{H}}_{\alpha}$ be its formal completion $\prod_{m \in \mathbb{Z}} {\rm{H}}^{m} (\mathfrak{N}_{\alpha},\mathbf{k}) $ and $\mathcal{H}^{0}_{\alpha}={\rm{H}}^{0}(\mathfrak{N}_{\alpha},\mathbf{k})  $ be its zero degree, then the diagram in \ref{KHA} also induces a product $ \odot$ on $\widehat{\mathcal{H}}=\bigoplus\limits_{\alpha \in \mathbb{N}Q_{0}}\widehat{\mathcal{H}}_{\alpha}$,

$$ \odot: \widehat{\mathcal{H}}_{\alpha} \times \widehat{\mathcal{H}}_{\beta} \rightarrow \widehat{\mathcal{H}}_{\gamma},   c_{1} \odot c_{2} =  \tilde{q}_{\ast}\tilde{p}^{!}(c_{1} \boxtimes c_{2}).  $$
The $\odot$ restricts to  products on $\mathcal{H}=\bigoplus \limits_{\alpha \in \mathbb{N}Q_{0}}$ and $\mathcal{H}^{0}=\bigoplus\limits_{\alpha \in \mathbb{N}Q_{0}} \mathcal{H}^{0}_{\alpha}$ respectively, and we can also define a $\Psi$-twisted version product $\odot^{\Psi}$ as what we have done for $\star^{\Psi}$.

\begin{proof}
	Let $\mathcal{Q}_{\alpha,c}$ be the full subcategory of the equivariant derived category $\mathcal{D}^{b}_{\mathbf{G}_{\alpha},c}(\mathbf{E}_{\alpha},\mathbf{k})$ of constructible sheaves consisting of $rat(\mathcal{M})$ with $\mathcal{M}$ in $\mathcal{Q}_{\alpha}$ and let $\mathcal{K}'_{\oplus} ( \mathcal{Q}_{\alpha,c} )$ be the split Grothendieck group of $\mathcal{Q}_{\alpha,c}$ such that $[L[1]]=-v[L]$, then $\star^{\Psi}$ induces a product on the split Grothendieck group $\mathcal{K}'_{\oplus}(\mathcal{Q}_{c})=\bigoplus\limits_{\alpha \in \mathbb{N}Q_{0}}\mathcal{K}'_{\oplus}(\mathcal{Q}_{\alpha,c})$ and $rat$ induces a $\mathbb{Z}[v^{\pm}]$ linear map $[rat]:\mathcal{K}'_{\oplus}(\mathcal{Q}_{\alpha}) \rightarrow \mathcal{K}'_{\oplus}(\mathcal{Q}_{\alpha},c)$. Following \cite{HLCC}, there is a natural characteristic cycle map $$CC:   \mathcal{K}'_{\oplus}(\mathcal{Q}_{\alpha},c)|_{v=1} \rightarrow  {\rm{H}}^{BM}_{top} (\mathfrak{N}_{\alpha},\mathbf{k}).$$
	Let $Ch_{0}: \mathbb{Q}\otimes_{\mathbb{Z}}\mathcal{A}^{nil}_{\alpha} \rightarrow \mathcal{H}^{0}_{\alpha}  $ be the composition of ${\rm{For}}^{\mathbb{C}^{\times}}$ and zero degree Chern character,
    then we claim that we have the following commutative diagram of $\mathbb{Q}$-algebras
    \[
    \xymatrix{
   ( \mathbb{Q}\otimes_{\mathbb{Z}}\mathcal{K}'_{\oplus}(\mathcal{Q})|_{v=1},\star^{\Psi}) \ar[r]^{[gr]} \ar[d]_{[rat]} &  ( \mathbb{Q}\otimes_{\mathbb{Z}}\mathcal{A}^{nil}|_{v=1}, \ast ) \ar[d]^{Ch_{0}} \\
   ( \mathbb{Q}\otimes_{\mathbb{Z}}\mathcal{K}'_{\oplus}(\mathcal{Q}_{c})|_{v=1},\star^{\Psi}) \ar[r]^{CC}  & (\mathcal{H}^{0} ,\odot^{\Psi} ).
    }    
    \]
    Replace $[T^{\ast}_{\mathbf{E}_{\alpha}} \mathbf{E}_{\alpha}]$  of \cite[Theorem 9.11]{HLCC} by  $[T^{\ast}_{0} \mathbf{E}_{\alpha}]$, one can show that $CC$ is an isomorphism of algebras in a similar way.
    
    Following \cite[Theorem 4.8]{compar}, there is a morphism of $\mathbb{Q}$ algebras from $ \mathbb{Q}\otimes_{\mathbb{Z}}\mathcal{A}^{nil}|_{v=1}$ to twisted $\widehat{\mathcal{H}}$ defined by
    $$  \mathbb{Q}\otimes_{\mathbb{Z}}\mathcal{A}^{nil}_{\alpha} \rightarrow \widehat{\mathcal{H}}_{\alpha},  f \mapsto  Td^{1/2} _{\mathfrak{M}_{\alpha}  } Ch(f), $$  
    where $Ch$ is the Chern character and $Td^{1/2} _{\mathfrak{M}_{\alpha}  }$ is a square root of the Todd class of $\mathfrak{M}_{\alpha}$. The twist product of $\widehat{\mathcal{H}}$ is defined by $$\widehat{\mathcal{H}}_{\alpha} \times \widehat{\mathcal{H}}_{\beta} \rightarrow \widehat{\mathcal{H}}_{\gamma}, (c_{1},c_{2} ) \mapsto t_{\alpha,\beta}c^{\alpha}_{\beta}\tilde{q}_{*} \tilde{p}^{!}(c^{\beta}_{\alpha})^{-1} (c_{1}\boxtimes c_{2}),  $$
    where $t_{\alpha,\beta}$ is the Chern character of the twisting factor in equation \ref{defKHA} and $c^{\beta}_{\alpha}, c^{\alpha}_{\beta}$ are certain classes. For the definition of  $c^{\beta}_{\alpha}, c^{\alpha}_{\beta}$, one can see details in \cite[Section 3.4 and 4.3]{compar}. Since the statement of \cite{compar} holds for quivers with potentials, we use the fact \cite[Appendix]{Dav17} that the CoHA of preprojective algebras is isomorphic to the CoHA of  triple quiver with a certain potential.
    
    After restricting $f \mapsto  Td^{1/2} _{\mathfrak{M}_{\alpha}  } Ch(f)$
    to zero degree, we obtain the morphism $Ch_{0}$ since the Todd class contributes $1$ at zero degree. Noticing that $t_{\alpha,\beta}$ coincides with the $\Psi$-twist, we can see that $Ch_{0}$ is a morphism of algebras.
    
    In order to show that the diagram is commutative, it suffices to check for any $a \geqslant 1$ and $i \in Q_{0}$,
    $$CC([rat] [\mathbf{1}_{ai}] ) = Ch_{0} ([gr] [\mathbf{1}_{ai}] ).$$ By direct calculation,  $CC([rat][\mathbf{1}_{ai}])=[T^{\ast}_{0} \mathbf{E}_{ai} ]$. Here we identify ${\rm{H}}^{BM}_{top,\mathbf{G}_{ai}}(\Lambda_{ai},\mathbf{k})$ with $\mathcal{H}^{0}_{ai}$ via Poincare duality and $[T^{\ast}_{0} \mathbf{E}_{ai} ]$ means the fundamental class of the irreducible component $T^{\ast}_{0} \mathbf{E}_{ai}$ of $\Lambda_{ai}$. Similarly, $Ch_{0} ([gr] [\mathbf{1}_{ai}]) =Ch_{0} (\pi^{\ast}_{\mathbf{E}_{ai}} \mathcal{O}_{0} ), $ where $\pi_{\mathbf{E}_{ai}}:\overline{\mathbf{E}}_{ai} \rightarrow \mathbf{E}_{ai}$ is the projection and $\mathcal{O}_{0}$ is the skyscraper sheaf at $0 \in \mathbf{E}_{ai}$. Hence $Ch_{0}([gr] [\mathbf{1}_{ai}])= Ch_{0}( \mathcal{O}_{\pi^{-1}_{\mathbf{E}_{ai}}}(0) )= [T^{\ast}_{0} \mathbf{E}_{ai} ]$ via Poincare duality.
    
    By \cite[Theorem 9.11]{HLCC}, we can see that $Ch_{0}[gr]$ is an isomorphism of $\mathbb{Z}$-algebra at $v=1$, hence $[gr]$ is injective.

    Let $\mathcal{A}^{nil}_{zs}$ be the subalgebra of $\mathcal{A}^{nil}$ generated by elements  $[\mathcal{O}_{T^{\ast}_{0} \mathbf{E}_{ai}}], a\in \mathbb{N}, i \in Q_{0}$ and we call $\mathcal{A}^{nil}_{ss}$ the zero spherical subalgebra of the K-theoretic Hall algebra. Then we have the following corollary.
    \begin{corollary}\label{ciso}
    	There is a commutative diagram of isomorphisms of $\mathbb{Z}[v^{\pm}]$-algebras
    	\[\xymatrix{
    	& \mathbf{U}^{\mathbb{Z}}_{v}(\mathfrak{n}^{+}_{Q})  \ar[ld]_{\phi} \ar[dr]^{\psi} & \\
       (\mathcal{K}'_{\oplus}(\mathcal{Q}),\star^{\Psi}) \ar[rr]^{[gr]} 	&  & (\mathcal{A}^{nil}_{zs},\ast),
    	}\]
    	where $\phi$ is the unique morphism determined by $E^{(a)}_{i}\mapsto [\mathbf{1}_{ai}]$ if $i \in Q^{re}_{0}$ and $E_{i,a} \mapsto [\mathbf{1}_{ai}]$ if $i \in Q^{im}_{0}$, and $\psi$ is the unique morphism determined by $E^{(a)}_{i}\mapsto [\mathcal{O}_{T^{\ast}_{0} \mathbf{E}_{ai}}]$ if $i \in Q^{re}_{0}$ and $E_{i,a} \mapsto [\mathcal{O}_{T^{\ast}_{0} \mathbf{E}_{ai}}]$ if $i \in Q^{im}_{0}$. Here $a \in \mathbb{N}$.
    \end{corollary}
   
\end{proof}

\section{Reflection functors and associated graded modules}
In this section, we fix a vertex $i\in Q^{re}_{0}$ and take $Q=(Q_{0},Q_{1},s,t)$ such that $i$ is a source in $Q_{0}$, that is, $t(e) \neq i$ for any $e \in Q_{1}$. Let $Q'=\sigma_{i}Q$ be the quiver obtained by reversing all arrows adjacent to $i$, then $s(e) \neq i$ for any $e \in Q'_{1}$ and we say $i$ is a sink in $Q'$.
\subsection{Geometric BGP reflection functor for quivers}
Assume that $(\mathbf{V},x)=(\bigoplus\limits_{i \in Q_{0}}\mathbf{V}_{i},(x_{e})_{e\in Q_{1}})$ is a representation of $Q$ such that the simple representation at vertex $i \in Q_{0}$ doesn't appear as a direct summand of $(\mathbf{V},x)$, then the BGP reflection $\sigma_{i}(\mathbf{V},x)=(\mathbf{V}',x')$ is defined to be the representation of $Q'$ given by
\begin{enumerate}
	\item For any $j\neq i \in Q_{0}$, set $\mathbf{V}'_{j} =\mathbf{V}_{j}$ and $\mathbf{V}'_{i}= {\rm{Coker}} (\bigoplus\limits_{e \in Q_{1},s(e)=i}x_{e}: \mathbf{V}_{i} \rightarrow \bigoplus\limits_{e \in Q_{1},s(e)=i} \mathbf{V}_{t(e)}) ;$
	\item  For any $e \in Q_{1} \cap Q'_{1}$, set $x'_{e}=x_{e}$ and those $x'_{e}$ such that $t(e)=i$ are determined by the natural quotient map from $\bigoplus\limits_{e \in Q_{1},s(e)=i} \mathbf{V}_{t(e)}$ to the cokernel $\mathbf{V}'_{i}$.
\end{enumerate}
The BGP reflection induces Lusztig symmetries of half quantum groups on split Grothendieck groups. More precisely, since $i$ is a source in $Q$, there is a partition of $\mathbf{E}_{\alpha,Q}= \coprod\limits_{k \geqslant 0} \mathbf{E}^{(i)}_{\alpha,Q,k} $ for any $\alpha \in \mathbb{N}Q_{0}$ such that 
$$ \mathbf{E}^{(i)}_{\alpha,Q,k}=\{x \in \mathbf{E}_{\alpha,Q}| {\rm{dim}} ({\rm{Ker}} \bigoplus\limits_{e \in Q_{1},s(e)=i}x_{e}:\mathbb{C}^{\alpha_{i}} \rightarrow \bigoplus\limits_{e \in Q_{1},s(e)=i} \mathbb{C}^{\alpha_{t(e)}} ) =k.   \}   $$ 
Similarly, $\mathbf{E}_{\alpha,Q'}$ also admits a partition $\mathbf{E}_{\alpha,Q'}= \coprod\limits_{k \geqslant 0} {^{(i)}\mathbf{E}}_{\alpha,Q',k} $ such that 
$$ {^{(i)}\mathbf{E}}_{\alpha,Q',k}=\{x \in \mathbf{E}_{\alpha,Q'}| {\rm{dim}} ({\rm{Coker}} \bigoplus\limits_{e \in Q'_{1},t(e)=i}x_{e}: \bigoplus\limits_{e \in Q'_{1},t(e)=i} \mathbb{C}^{\alpha_{s(e)}} \rightarrow \mathbb{C}^{\alpha_{i}} ) =k.   \}   $$

Let $\mathcal{P}^{(i)}_{\alpha,Q}$ be the subset of $\mathcal{P}_{\alpha,Q}$ consisting of objects  whose supports intersect $ \mathbf{E}^{(i)}_{\alpha,Q,0}$ nontrivially, and let $\mathcal{Q}^{(i)}_{\alpha,Q}$ be the full subcategory of  $\mathcal{Q}_{\alpha,Q}$ consisting of complexes of the form $$\bigoplus\limits_{L \in \mathcal{P}^{(i)}_{\alpha,Q},k \in \mathbb{Z}}  L^{\oplus m_{L,k} }[k](\frac{k}{2})$$ such that $m_{L,k} \in \mathbb{N}$ and only finitely many $m_{L,k}$ are nonzero. Let $ \iota^{(i)} :\mathbf{E}^{(i)}_{\alpha,Q,0} \rightarrow \mathbf{E}_{\alpha,Q} $ be the open inclusion, then we have the following commutative diagram
\[
\xymatrix{
\mathcal{K}'_{\oplus} (\mathcal{Q}_{Q}) \ar[r]^{(\iota^{(i)})^{\ast}}  \ar[d]_{\phi^{-1}_{Q}}& \mathcal{K}'_{\oplus}(\mathcal{Q}^{(i)}_{Q}) \ar[d]^{\phi^{-1}_{Q}} \\
\mathbf{U}^{\mathbb{Z}}_{v}(\mathfrak{n}^{+}_{Q} ) \ar[r] &  \mathbf{U}^{\mathbb{Z},(i)}_{v}(\mathfrak{n}^{+}_{Q} ),
}
\] 
where $\mathbf{U}^{\mathbb{Z},(i)}_{v}(\mathfrak{n}^{+}_{Q} )= \mathbf{U}^{\mathbb{Z}}_{v}(\mathfrak{n}^{+}_{Q} )/\sum\limits_{a\geqslant 1}\mathbf{U}^{\mathbb{Z}}_{v}(\mathfrak{n}^{+}_{Q} )E^{(a)}_{i}$ and the bottom arrow is the quotient map. Here we identify $  \mathcal{K}'_{\oplus}(\mathcal{Q}^{(i)}_{Q}) $ with $ \mathcal{K}'_{\oplus}( (\iota^{(i)})^{\ast} \mathcal{Q}_{Q})$, since they are isomorphic to each other as $\mathbb{Z}[v^{\pm}]$-modules. Replace  $\iota^{(i)}:\mathbf{E}^{(i)}_{\alpha,Q,0} \rightarrow \mathbf{E}_{\alpha,Q} $ by ${^{(i)} \iota }: {^{(i)}\mathbf{E}}_{\alpha,Q',0} \rightarrow \mathbf{E}_{\alpha,Q'}$, one can define ${^{(i)} \mathcal{Q}_{\alpha,Q'}}$ in a similar way and get the following commutative diagram
\[
\xymatrix{
	\mathcal{K}'_{\oplus} (\mathcal{Q}_{Q'}) \ar[r]^{(^{(i)}\iota)^{\ast}}  \ar[d]_{\phi^{-1}_{Q'}}& \mathcal{K}'_{\oplus}(^{(i)}\mathcal{Q}_{Q'}) \ar[d]^{\phi^{-1}_{Q'}} \\
	\mathbf{U}^{\mathbb{Z}}_{v}(\mathfrak{n}^{+}_{Q'} ) \ar[r] &  ^{(i)}\mathbf{U}^{\mathbb{Z}}_{v}(\mathfrak{n}^{+}_{Q'} ),
}\] 
where $^{(i)}\mathbf{U}^{\mathbb{Z}}_{v}(\mathfrak{n}^{+}_{Q'} )= \mathbf{U}^{\mathbb{Z}}_{v}(\mathfrak{n}^{+}_{Q'} )/\sum\limits_{a\geqslant 1}E^{(a)}_{i}\mathbf{U}^{\mathbb{Z}}_{v}(\mathfrak{n}^{+}_{Q'} )$.

Assume that $\beta=s_{i}(\alpha)$, let $Z_{\alpha,\beta}$ be the subvariety of $\mathbf{E}_{\alpha,Q} \times \mathbf{E}_{\beta,Q'}$ whose close points are those $(x_{e})_{e \in Q_{1}} ,(x'_{e})_{e\in Q'_{1}}$ such that $x_{e}=x'_{e}$ if $e \in Q_{1} \cap Q'_{1}$, and $(x_{e})_{e \in Q_{1},s(e)=i} ,  (x'_{e})_{e \in Q'_{1},t(e)=i}  $ fit into the following short exact sequence
$$  0 \longrightarrow \mathbb{C}^{\alpha_{i}} \xrightarrow{x^{(i)}}    \bigoplus\limits_{e \in Q_{1},s(e)=i}\mathbb{C}^{\alpha_{t(e)}} = \bigoplus\limits_{e \in Q'_{1},t(e)=i}\mathbb{C}^{\beta_{s(e)}} \xrightarrow{{^{(i)}x'}}  \mathbb{C}^{\beta_{i}} \longrightarrow 0,  $$
where $x^{(i)}=\bigoplus\limits_{e \in Q_{1},s(e)=i}x_{e} $ and $^{(i)}x'=\bigoplus\limits_{e \in Q'_{1},t(e)=i}x'_{e}$.
The algebraic group $\mathbf{G}_{\alpha,\beta}= GL(\mathbb{C}^{\alpha_{i}}) \times GL(\mathbb{C}^{\beta_{i}})  \times  \prod_{j \neq i \in Q_{0}} GL(\mathbb{C}^{\alpha_{j}})  $ acts naturally on $Z_{\alpha,\beta}$.
By restricting projections $pr_{1},pr_{2}: \mathbf{E}_{\alpha,Q} \times \mathbf{E}_{\beta,Q'} \rightarrow \mathbf{E}_{\alpha,Q}, \mathbf{E}_{\beta,Q'} $ to $Z_{\alpha,\beta}$, we obtain a $GL(\mathbb{C}^{\beta_{i}})$-principal bundle $pr_{1}:Z_{\alpha,\beta} \rightarrow  \mathbf{E}^{(i)}_{\alpha,Q,0}$ and a $GL(\mathbb{C}^{\alpha_{i}})$-principal bundle 
$pr_{2}:Z_{\alpha,\beta} \rightarrow {^{(i)} \mathbf{E}_{\beta,Q',0}}$,
\[
\xymatrix{
& Z_{\alpha,\beta} \ar[dl]_{pr_{1}} \ar[dr]^{pr_{2}} &\\
\mathbf{E}^{(i)}_{\alpha,Q,0}& &{^{(i)} \mathbf{E}_{\beta,Q',0}}}.
\] 
Therefore we can define the following geometric BGP reflection functor
$$\mathbf{S}_{i} : \mathcal{D}^{b}_{\mathbf{G}_{\alpha}}({\rm{MHM}}(\mathbf{E}^{(i)}_{\alpha,Q,0},\mathbf{k}) ) \rightarrow  \mathcal{D}^{b}_{\mathbf{G}_{\beta}}({\rm{MHM}}({^{(i)}\mathbf{E}}_{\beta,Q',0},\mathbf{k}) ),$$ $$\mathcal{M} \mapsto (pr^{\ast}_{2} {\rm{Infl}}^{\mathbf{G}_{\alpha,\beta}}_{\mathbf{G}_{\beta}} )^{-1} pr^{\ast}_{1} {\rm{Infl}}^{\mathbf{G}_{\alpha,\beta}}_{\mathbf{G}_{\alpha}} \mathcal{M}[\alpha^{2}_{i}-\beta^{2}_{i}](\frac{1}{2}(\alpha^{2}_{i}-\beta^{2}_{i})  ). $$
which is a derived equivalence and induces a $\mathbb{Z}[v^{\pm}]$-linear isomorphism 
$[\mathbf{S}_{i}]:\mathcal{K}'_{\oplus} (\mathcal{Q}^{(i)}_{Q}) \rightarrow \mathcal{K}'_{\oplus} (^{(i)}\mathcal{Q}_{Q'}). $ 
By a similar argument as \cite[II.7]{Coha26} or \cite{BGP}, one can show that $\mathbf{S}_{i}$ fits into the following commutative diagram
\[
\xymatrix{
\mathcal{K}'_{\oplus} (\mathcal{Q}^{(i)}_{Q}) \ar[r]^{[\mathbf{S}_{i}]} \ar[d]_{\phi^{-1}_{Q}} & \mathcal{K}'_{\oplus} (^{(i)}\mathcal{Q}_{Q'}) \ar[d]^{\phi^{-1}_{Q'}} \\
 \mathbf{U}^{\mathbb{Z},(i)}_{v}(\mathfrak{n}^{+}_{Q} ) \ar[r]^{T''_{i,1}}   & ^{(i)}\mathbf{U}^{\mathbb{Z}}_{v}(\mathfrak{n}^{+}_{Q'} ),
}
\]
 denote the inverse of $\mathbf{S}_{i}$ by $\mathbf{S}'_{i}$, then there is another commutative diagram 
 \[
 \xymatrix{
 	\mathcal{K}'_{\oplus} (^{(i)}\mathcal{Q}_{Q'}) \ar[r]^{[\mathbf{S}'_{i}]}  \ar[r]^{[\mathbf{S}_{i}]} \ar[d]_{\phi^{-1}_{Q'}} &  \mathcal{K}'_{\oplus} (\mathcal{Q}^{(i)}_{Q})  \ar[d]^{\phi^{-1}_{Q}} \\
 ^{(i)}\mathbf{U}^{\mathbb{Z}}_{v}(\mathfrak{n}^{+}_{Q'} )	 \ar[r]^{T'_{i,-1}}   & \mathbf{U}^{\mathbb{Z},(i)}_{v}(\mathfrak{n}^{+}_{Q} ),
 }
 \]
 where $T'_{i,-1},T''_{i,1}$ are Lusztig's symmetries defined in \cite[Chapter 38]{lusztig2010introduction}.
 
\subsection{Derived reflection functor for preprojective algebra}
Let $I_{i}= \Pi_{Q}(1-e_{i}) \Pi_{Q}$ be the ideal of $\Pi_{Q}$, where $e_{i}$ is the idempotent associated to the vertex $i$, then it is a tilting $(\Pi_{Q},\Pi_{Q})$-bimodule such that ${\rm{End}}(I_{i})=\Pi_{Q}$. There are torsion pairs $(\mathcal{T}_{i},\mathcal{F}_{i})$ and $(\mathcal{T}^{i},\mathcal{F}^{i})$ induced from $I_{i}$ given by
$$ \mathcal{T}_{i}=\{ M \in {\rm{Mod}}(\Pi_{Q}) , I_{i} \otimes_{\Pi_{Q}}(M)=0 \}, $$
$$ \mathcal{F}_{i}=\{ M \in {\rm{Mod}}(\Pi_{Q}) ,{\rm{Tor}}^{1}_{\Pi_{Q}} (I_{i},M)=0 \}, $$
$$ \mathcal{T}^{i}=\{ M \in {\rm{Mod}}(\Pi_{Q}) ,{\rm{Ext}}^{1}_{\Pi_{Q}} (I_{i},M)=0 \}, $$
$$ \mathcal{F}^{i}=\{ M \in {\rm{Mod}}(\Pi_{Q}) ,{\rm{Hom}}_{\Pi_{Q}} (I_{i},M)=0 \}. $$
Let $R_{i}=I_{i} \otimes_{\Pi_{Q}}(-)$ and $R'_{i}={\rm{Hom}}_{\Pi_{Q}}(I_{i},-)$, then by \cite[Theorem 5.4]{BKT14} and \cite[Proposition 2.7]{SY13} tilting theory provides mutually inverse equivalences  
\[
\xymatrix{
	\mathcal{F}_{i} \ar@<0.5ex>[r]^{R_{i}} & \mathcal{T}^{i} \ar@<0.5ex>[l]^{R'_{i}}
}.
\]

Let $\Lambda^{(i)}_{\alpha,Q}$ be the open subvariety of $\Lambda_{\alpha,Q}$ consisting of $(x,\bar{x})$ such that $x^{(i)}= \bigoplus\limits_{e \in Q_{1},s(e)=i} x_{e}$ is injective, then the stack $\mathfrak{N}^{(i)}_{\alpha,Q}=\Lambda^{(i)}_{\alpha,Q}/\mathbf{G}_{\alpha}$ parameterizes the representations of $\Pi_{Q}$ in $\mathcal{F}_{i}$ with dimension vector $\alpha$.
Similarly, let $^{(i)}\Lambda_{\alpha,Q'}$ be the open subvariety of $\Lambda_{\alpha,Q'}$ consisting of $(x',\bar{x}')$ such that $^{(i)}x'= \bigoplus\limits_{e \in Q_{1},t(e)=i} x'_{e}$ is surjective, then the stack $^{(i)}\mathfrak{N}_{\alpha,Q'}={^{(i)}\Lambda}_{\alpha,Q'}/\mathbf{G}_{\alpha}$ parameterizes the representations of $\Pi_{Q'}=\Pi_{Q}$ in $\mathcal{F}^{i}$ with dimension vector $\alpha$. Here the subscript $Q'$ of $\Lambda_{\alpha,Q'}$ means that $\mathbb{C}^{\times}$ acts trivially on $x'_{e}, e\in Q'_{1}$ and scales $\bar{x}'_{e^{\ast}},e^{\ast} \in Q^{',\ast}_{1}$ with weight $-1$. Following \cite[2.2]{BK12}, the equivalence $R_{i}$ induces  an isomorphism between moduli stacks $\mathfrak{N}^{(i)}_{\alpha,Q}$ and $^{(i)}\mathfrak{N}_{\beta,Q'}$ such that $s_{i}(\alpha)=\beta$, which is given by the following diagram
\[
\xymatrix{ & \overline{Z}_{\alpha,\beta} \ar[dr]^{\rho_{2}} \ar[dl]_{\rho_{1}} &\\
\Lambda^{(i)}_{\alpha,Q}	& & ^{(i)}\Lambda_{\beta,Q'}.
}\]  
Here the variety $\overline{Z}_{\alpha,\beta}$ consists of $(x,\bar{x},\rho)$, where $(x,\bar{x}) \in \Lambda^{(i)}_{\alpha,Q}$ and $\rho:  {\rm{Coker}}(x^{(i)} ) \cong \mathbb{C}^{\beta_{i}} $ is a linear isomorphism, the map $\rho_{1}$ is the projection and the map $\rho_{2} (x,\bar{x},\rho) =(x',\bar{x}')$ is determined by the following conditions
\begin{enumerate}
	\item If  $e\in Q_{1} \cap Q'_{1}$, set $x'_{e}=x_{e}$, and if $e^{\ast} \in Q^{\ast}_{1} \cap Q^{',\ast}_{1}$, set  $\bar{x}'_{e^{\ast}}=\bar{x}_{e^{\ast}}$.
	\item If $e\in Q'_{1}$ such that $t(e)=i$, set $x'_{e}$ be the composition $$\mathbb{C}^{\beta_{s(e)}} \rightarrow  \bigoplus\limits_{e \in Q'_{1},t(e)=i} \mathbb{C}^{\beta_{s(e)}}=\bigoplus \limits_{e \in Q_{1},s(e)=i} \mathbb{C}^{\alpha_{t(e)}} \rightarrow  {\rm{Coker}}(x^{(i)} ) \xrightarrow{\rho} \mathbb{C}^{\beta_{i}},$$
	where the first arrow is the natural inclusion and the second arrow is the quotient map. 
	\item The preprojective relation implies that the composition ${^{(i)}\bar{x}} \circ x^{(i)} \in {\rm{End}}(\mathbb{C}^{\alpha_{i}}) $ of $ x^{(i)}$ and  ${^{(i)}\bar{x}}=\bigoplus\limits_{e \in Q^{\ast}_{1},t(e^{\ast})=i}\bar{x}_{e^{\ast}} : \bigoplus\limits_{e \in Q^{\ast}_{1},t(e^{\ast})=i} \mathbb{C}^{\alpha_{s(e^{\ast})}} \rightarrow  \mathbb{C}^{\alpha_{i}}$ is zero, hence $ {^{(i)}\bar{x}} $ induces a morphism $  {\rm{Coker}}(x^{(i)} ) \rightarrow \mathbb{C}^{\alpha_{i}}$, which is still denoted by $ {^{(i)}\bar{x}} $. If $e^{\ast} \in Q^{',\ast}_{1}$ such that $s(e^{\ast})=i$, set $x'_{e^{\ast}}$ be the composition
	$$ \mathbb{C}^{\beta_{i}} \xrightarrow{ \rho^{-1}} {\rm{Coker}}(x^{(i)} ) \xrightarrow{ {^{(i)}\bar{x}}} \mathbb{C}^{\alpha_{i}} \xrightarrow{ x^{(i)}} \bigoplus\limits_{e \in Q_{1},s(e)=i} \mathbb{C}^{\alpha_{t(e)}} =\bigoplus\limits_{e \in Q^{',\ast}_{1},s(e^{\ast})=i} \mathbb{C}^{\beta_{t(e^{\ast})}}   \rightarrow \mathbb{C}^{\beta_{t(e^{\ast})}}, $$
	where the last arrow is the natural projection. 
\end{enumerate} 
In particular, the preprojective relation still holds for $(x',\bar{x}')$ and $^{(i)}x'$ is surjective by definition. It's easy to see that $\rho_{1}$ is a principal $GL(\mathbb{C}^{\beta_{i}})$-bundle, $\rho_{2}$ is a principal $GL(\mathbb{C}^{\alpha_{i}})$-bundle and both $\rho_{1}$ and $\rho_{2}$ are $\mathbb{C}^{\times}$-equivariant. The diagram above defines a derived equivalence 
$$ \mathbf{R}_{i}: \mathcal{D}^{b}( {\rm{Qcoh}}^{\mathbf{G}_{\alpha} \times \mathbb{C}^{\times}}(\Lambda^{(i)}_{\alpha,Q})    ) \rightarrow \mathcal{D}^{b}( {\rm{Qcoh}}^{\mathbf{G}_{\beta} \times \mathbb{C}^{\times}}(^{(i)}\Lambda_{\beta,Q'})    ),$$  $$\mathcal{M} \mapsto  (\rho^{\ast}_{2} {\rm{Infl}}^{\mathbf{G}_{\alpha,\beta}}_{\mathbf{G}_{\beta}} )^{-1} \rho^{\ast}_{1} {\rm{Infl}}^{\mathbf{G}_{\alpha,\beta}}_{\mathbf{G}_{\alpha}} \mathcal{M}[2(\alpha^{2}_{i}- \beta^{2}_{i})] \langle \beta^{2}_{i}-\alpha^{2}_{i} \rangle, $$
which induces a linear map $[\mathbf{R}_{i}]: \mathcal{A}^{nil,(i)}_{\alpha} \rightarrow {^{(i)}\mathcal{A}}^{nil}_{\beta}. $ Here $ \mathcal{A}^{nil,(i)}_{\alpha}= K_{\mathbb{C}^{\times}}(\mathfrak{N}^{(i)}_{\alpha})$ and ${^{(i)}\mathcal{A}}^{nil}_{\beta} =K_{\mathbb{C}^{\times}}({^{(i)}\mathfrak{N} }_{\beta} ). $ 

\subsection{Compare reflection functors}
Consider the full subcategory $\mathcal{D}^{b}_{\mathbf{G}_{\alpha}}({\rm{MHM}} (\mathbf{E}^{(i)}_{\alpha,Q,0},\Lambda^{(i)}_{\alpha,Q},\mathbf{k} ))'$ of $\mathcal{D}^{b}_{\mathbf{G}_{\alpha}}({\rm{MHM}} (\mathbf{E}^{(i)}_{\alpha,Q,0},\mathbf{k} ))'$ and denote it by $\mathcal{D}^{nil,(i)}_{\alpha}$, and similarly denote the full subcategory 
$\mathcal{D}^{b}_{\mathbf{G}_{\alpha}}({\rm{MHM}} ({^{(i)}\mathbf{E}}_{\alpha,Q',0},{^{(i)}\Lambda}_{\alpha,Q'},\mathbf{k} ))'$ by $^{(i)} \mathcal{D}^{nil}_{\alpha}$.  
\begin{proposition}
	The $\mathbb{Z}[v^{\pm}]$-linear map $[gr]$ intertwines the geometric BGP reflection functors $[\mathbf{S}_{i}]$ and the derived reflection functors $[\mathbf{R}_{i}]$. More precisely, assume $s_{i}(\alpha)=\beta \in \mathbb{N}Q_{0}$, then there is a commutative diagram
	\[
	\xymatrix{
  \mathcal{K}_{0}(\mathcal{D}^{nil,(i)}_{\alpha})  \ar[r]^{[\mathbf{S}_{i}]} \ar[d]_{[gr]}	& \mathcal{K}_{0}({^{(i)} \mathcal{D}}^{nil}_{\beta}) \ar[d]^{[gr]} \\
\mathcal{A}^{nil,(i)}_{\alpha} \ar[r]^{[\mathbf{R}_{i}]}	&  {^{(i)}\mathcal{A}}^{nil}_{\beta}. }
	\]
\end{proposition}
\begin{proof}
	We claim that there exists an isomorphism $  \Lambda_{\alpha,Q} \times_{\mathbf{E}_{\alpha,Q}} Z_{\alpha,\beta} \cong \overline{Z}_{\alpha,\beta} $. Indeed, denote an element of $\Lambda_{\alpha,Q} \times_{\mathbf{E}_{\alpha,Q}} Z_{\alpha,\beta}$ by $(x,\bar{x},x')$ such that $(x,\bar{x}) \in \Lambda_{\alpha,Q}$ and $x' \in \mathbf{E}_{\beta,Q'}$, the morphism $x^{(i)}$ and $^{(i)}x'$ form an exact sequence implies that $^{(i)}x'$ induces a surjective linear map $^{(i)}\tilde{x}':{\rm{Coker}} (x^{(i)}) \rightarrow \mathbb{C}^{\beta_{i}}$. Since ${\rm{Coker}} (x^{(i)})$ has dimension $\beta_{i}$, we know that  $^{(i)}\tilde{x}'$ is a linear isomorphism. Conversely, given $(x,\bar{x}, \rho) \in \overline{Z}_{\alpha,\beta}$, let $^{(i)}x'$ be the composition $ \bigoplus\limits_{e \in Q_{1},s(e)=i}\mathbb{C}^{\alpha_{t(e)}} \rightarrow {\rm{Coker}}(x^{(i)}) \xrightarrow{\rho} \mathbb{C}^{\beta_{i}} $, then $^{(i)}x'$ and $x^{(i)}$ form a short exact sequence, and this $^{(i)}x'$ (together with $x$) determines an element $(x,x')$ of $Z_{\alpha,\beta}$. One can check by definition that $(x,\bar{x},x') \mapsto (x,\bar{x}, {^{(i)}\tilde{x}}') $
	and $(x,\bar{x}, \rho) \mapsto (x,\bar{x},x') $ are mutually inverse. In particular, the cotangent correspondence of $pr_{1},pr_{2}$ fit into the diagram of $\rho_{1},\rho_{2}$.
 	
	By Proposition \ref{eqo}, we have the following equation
	\begin{equation}
		\begin{split}
		 &gr (\mathbf{S}_{i} \mathcal{M}) \\
		\cong & gr (pr^{\ast}_{2} {\rm{Infl}}^{\mathbf{G}_{\alpha,\beta}}_{\mathbf{G}_{\beta}} )^{-1} pr^{\ast}_{1} {\rm{Infl}}^{\mathbf{G}_{\alpha,\beta}}_{\mathbf{G}_{\alpha}} \mathcal{M}[\alpha^{2}_{i}-\beta^{2}_{i}](\frac{1}{2}(\alpha^{2}_{i}-\beta^{2}_{i})    \\
		\cong &  (\rho^{\ast}_{2} {\rm{Infl}}^{\mathbf{G}_{\alpha,\beta}}_{\mathbf{G}_{\beta}} )^{-1} (\rho^{\ast}_{1} {\rm{Infl}}^{\mathbf{G}_{\alpha,\beta}}_{\mathbf{G}_{\alpha}} gr(\mathcal{M}) \otimes p^{\ast}\omega_{pr_{1}} \otimes p^{\ast}\omega^{-1}_{pr_{2}} )[2(\alpha^{2}_{i}-\beta^{2}_{i})] \langle (-\alpha^{2}_{i}+\beta^{2}_{i})  \rangle,
		\end{split}
	\end{equation}
	where $p: \overline{Z}_{\alpha,\beta} \rightarrow  Z_{\alpha,\beta} $ is the projection $(x,\bar{x},x') \mapsto (x,x')$. (Here we identify $\overline{ Z}_{\alpha,\beta}$ with  $  \Lambda_{\alpha,Q} \times_{\mathbf{E}_{\alpha,Q}} Z_{\alpha,\beta} \cong \overline{Z}_{\alpha,\beta} $.) 
	
	It remains to show that $p^{\ast}\omega_{pr_{1}} \otimes p^{\ast}\omega^{-1}_{pr_{2}} $ is trivial. Since $p^{\ast}\omega_{pr_{1}} \otimes p^{\ast}\omega^{-1}_{pr_{2}} \cong p^{\ast} ( \omega_{pr_{1}} \otimes \omega^{-1}_{pr_{2}}) \cong p^{\ast}( pr^{\ast}_{1}\omega^{-1}_{\mathbf{E}^{(i)}_{\alpha,Q,0}} \otimes pr^{\ast}_{2} \omega_{{^{(i)}\mathbf{E}}_{\beta,Q',0}} )$, it suffices to  show that $pr^{\ast}_{1}\omega^{-1}_{\mathbf{E}^{(i)}_{\alpha,Q,0}} \otimes pr^{\ast}_{2} \omega_{{^{(i)}\mathbf{E}}_{\beta,Q',0}} $ is trivial. With the notations in the proof of Proposition \ref{inj}, the sheaf $pr^{\ast}_{1}\omega^{-1}_{\mathbf{E}^{(i)}_{\alpha,Q,0}}$ contributes $det ( \bigoplus\limits_{e \in Q_{1},s(e)=i}\mathcal{V}_{\alpha_{i}} \otimes \mathcal{V}^{\vee}_{\alpha_{t(e)} } ) $  and $det(\bigoplus\limits_{e \in Q_{1}\cap Q'_{i}}  \mathcal{V}_{\alpha_{s(e)}} \otimes \mathcal{V}^{\vee}_{\alpha_{t(e)}} )$, and $pr^{\ast}_{2} \omega_{{^{(i)}\mathbf{E}}_{\beta,Q',0}}$ contributes $det ( \bigoplus\limits_{e \in Q'_{1},t(e)=i}\mathcal{V}^{\vee}_{\beta_{s(e)}} \otimes \mathcal{V}_{\beta_{i} } ) $ and $det(\bigoplus\limits_{e \in Q_{1}\cap Q'_{i}}  \mathcal{V}^{\vee}_{\beta_{s(e)}} \otimes \mathcal{V}_{\beta_{t(e)}} )$, where $\vee$ means taking dual vector bundle. Notice that on $Z_{\alpha,\beta}$, $\mathcal{V}_{\alpha_{j}}=\mathcal{V}_{\beta_{j}}$ for any $j \neq i$, and $det(\mathcal{V}_{\alpha_{i}})\otimes det(\mathcal{V}_{\beta_{i}}) \cong det ( \bigoplus\limits_{e \in Q_{1},s(e)=i} \mathcal{V}_{\alpha_{t(e)} } )=det ( \bigoplus\limits_{e \in Q'_{1},t(e)=i}\mathcal{V}_{\beta_{s(e)}})$ since they form a short exact sequence, we finish the proof.
\end{proof}

Let $\tilde{\iota}^{(i)}: \Lambda^{(i)}_{\alpha} \rightarrow \Lambda_{\alpha}$ be the open inclusion and denote $(\tilde{\iota}^{(i)})^{\ast} \mathcal{A}^{nil}_{zs,\alpha}$ by $\mathcal{A}^{nil,(i)}_{zs,\alpha}$. Simiarly, one can define $^{(i)}\tilde{\iota}$ and $^{(i)}\mathcal{A}^{nil}_{zs,\beta}$. Apply Corollary \ref{ciso}, we get the following corollary.
\begin{corollary}
	\begin{enumerate}
		\item   The kernel of $(\tilde{\iota}^{(i)})^{\ast}:\mathcal{A}^{nil}_{zs,\alpha} \rightarrow \mathcal{A}^{nil,(i)}_{zs,\alpha}$ is $\sum\limits_{a= 1}\limits^{\alpha_{i}}\mathcal{A}^{nil}_{zs,\alpha-ai}\ast [\mathcal{O}_{T^{\ast}_{0} \mathbf{E}_{ai,Q} }] $, and the kernel of $({^{(i)}\tilde{\iota}})^{\ast}:\mathcal{A}^{nil}_{zs,\beta} \rightarrow {^{(i)}\mathcal{A}}^{nil}_{zs,\beta}$ is $ \sum\limits_{a= 1}^{\beta_{i}}[\mathcal{O}_{T^{\ast}_{0} \mathbf{E}_{ai,Q'} }] \ast \mathcal{A}^{nil}_{zs,\beta-ai}$.
		\item 	There is a commutative diagram  of isomorphisms of $\mathbb{Z}[v^{\pm}]$-modules
		\[
		\xymatrix{
			& \mathcal{K}'_{\oplus}( \mathcal{Q}^{(i)}_{\alpha,Q} ) \ar[r]^{[\mathbf{S}_{i}]} \ar[dd] &  \mathcal{K}'_{\oplus}({^{(i)}\mathcal{Q}}_{\beta,Q'}) \ar[dd] & \\
			\mathbf{U}^{\mathbb{Z},(i)}_{v} (\mathfrak{n}^{+}_{Q}) \ar[rrr]^{T''_{i,1}} \ar[ru]^{\phi_{Q}} \ar[rd]_{\psi_{Q}}	& & & ^{(i)}\mathbf{U}^{\mathbb{Z}}_{v} (\mathfrak{n}^{+}_{Q'}) \ar[lu]_{\phi_{Q'}} \ar[ld]^{\psi_{Q'}}  \\
			&  \mathcal{A}^{nil,(i)}_{zs,\alpha} \ar[r]^{\mathbf{R}_{i}} &  ^{(i)}\mathcal{A}^{nil}_{zs,\beta} &,
		}
		\]
		where  the vertical arrows are given by $[gr]$.
	\end{enumerate}
\end{corollary}
We remark that after taking $rat$ and $Ch_{0}$ as the proof of Proposition \ref{inj}, the corollary above degenerates to \cite[Corollary II.7.11]{Coha26} at $v=1$.

\end{spacing}

\bibliography{ref.bib}
\end{document}